\documentclass{amsart}

\usepackage{amssymb}
\usepackage{ifpdf}
\ifpdf
\usepackage[pdftex]{hyperref}
\else
\usepackage[dvips]{hyperref}
\fi

\makeatletter

 \theoremstyle{plain}    
 \newtheorem{thm}{Theorem}[section]
 \numberwithin{equation}{section} 
 \numberwithin{figure}{section} 
 \theoremstyle{plain}
 \theoremstyle{definition}
  \newtheorem{problem}[thm]{Problem}
 \theoremstyle{plain}    
 \theoremstyle{remark}
 \newtheorem{rem}[thm]{Remark}
 \theoremstyle{remark}    
  
 \theoremstyle{plain}    
  
 \theoremstyle{plain}    
 \newtheorem{lem}[thm]{Lemma} 
 \theoremstyle{definition}
 \newtheorem{defn}[thm]{Definition}
 \theoremstyle{plain}    
 \theoremstyle{plain}    
 \newtheorem{prop}[thm]{Proposition} 

\usepackage{mathrsfs}
\usepackage{dsfont}

\def\N{\mathbb {N}}
\def\Z{\mathbb {Z}}
\def\Q{\mathbb {Q}}
\def\R{\mathbb {R}}
\def\C{\mathbb {C}}

\def\isom{\simeq}

\def\bs{\backslash}

\def\supp{\qopname\relax o{supp}}

\newcommand{\order}{\mathcal{O}}

\DeclareMathOperator{\Gal}{Gal}
\def\Hom{\qopname\relax o{Hom}}

\renewcommand{\H}{\mathbf{H}}

\newcommand\vol{\mathrm{vol}}

\newcommand\abs[1]{\left|{#1}\right|}

\newcommand\dist{\mathrm{dist}}
\newcommand\distG{\dist_G}

\newcommand\xhookrightarrow[2][]{\ext@arrow 0062{\hookrightarrowfill@}{#1}{#2}}
\def\hookrightarrowfill@{\arrowfill@\lhook\relbar\rightarrow}

\newcommand\SL{\mathrm{SL}}
\newcommand\GL{\mathrm{GL}}

\newcommand\PGL{\mathrm{PGL}}

\newcommand\Gtwid{\tilde{\mathbf{G}}}

\newcommand{\adele}{\mathbb{A}}
\newcommand{\cmpct}{E}
\newcommand\PO{\mathrm{PO}}
\newcommand{\temp}{\mathrm{temp}}

\newcommand{\tube}{\mathcal{T}}

\def\cf{cf.\ }
\def\ie{i.e.\ }

\def\wrt{w.r.t.\ }

\newcommand{\D}{C^{\infty}_{\mathrm{c}}(X)}

\newcommand{\CX}{C^{\infty}(X)}

\newcommand\G{\mathbf{G}}
\newcommand\T{\mathbf{T}}

\newcommand{\barxB}{\overline{xB}}
\newcommand{\Hecke}{\mathcal{H}}
\newcommand\Af{\mathbb{A}_\mathrm{f}}

\newcommand{\GA}{\mathbf{G}(\mathbb{A})}
\newcommand{\GAf}{\mathbf{G}(\Af)}
\newcommand{\GQ}{\mathbf{G}(\Q)}

\newcommand{\Kf}{K_\mathrm{f}}

\newcommand{\xf}{x_\mathrm{f}}

\newcommand{\Ominf}{\Omega_\infty}

\newcommand{\denom}{\mathrm{d}}
\newcommand{\denp}{\widetilde{\denom}}
\makeatother
\begin{document}

\title[Entropy bounds\ldots]{Entropy bounds and quantum unique ergodicity for Hecke eigenfunctions on division algebras}

\author[L. Silberman]{Lior Silberman}
\address{Lior Silberman\\
  Department of Mathematics\\
  University of British Columbia\\
  Vancouver  BC  V6T 1Z2\\
  Canada.
}
\email{lior@math.ubc.edu}

\author[A. Venkatesh]{Akshay Venkatesh}
\address{Akshay Venkatesh\\
  Department of Mathematics\\
  Stanford University\\
  Stanford, CA, 94035\\
  USA.
}
\email{akshay@math.stanford.edu}

\begin{abstract}
We prove the arithmetic quantum unique ergodicity (AQUE) conjecture for
non-degenerate sequences of Hecke eigenfunctions on quotients
$\Gamma \backslash G/K$,  where $G\isom\PGL_{d}(\R)$, $K$
is a maximal compact subgroup of $G$ and $\Gamma<G$ is a lattice
associated to a division algebra over $\Q$ of prime degree $d$. 

More generally, we introduce a new method of proving positive entropy
of quantum limits, which applies to higher-rank groups.
The result on AQUE is obtained by combining this with a
measure-rigidity theorem due to Einsiedler-Katok,
following a strategy first pioneered by Lindenstrauss.
\end{abstract}

\maketitle

\tableofcontents

\section{Introduction}

\subsection{Result}
In this paper, we shall show (a slightly sharper version of)
the following statement.  For precise definitions we
refer to \S \ref{sec:results} especially \S \ref{finalresults}.
\begin{thm}\label{thm:intro}
Let $\Gamma$ be a lattice in $\PGL_d(\R)$, with $d$ prime,
associated to a division algebra\footnote{
  This means that $\Gamma$ is the image of $\order^{\times}$
  in $\PGL_d(\R)$, where $\order$ is an order in a $\Q$-division algebra
  so that $\order \otimes \R = M_d(\R)$. We also impose a class number one
  condition, see \S \ref{sec:results}.},
and $\psi_i$ a non-degenerate sequence of
Hecke-Maass eigenfunctions on
$Y := \Gamma \backslash \PGL_d(\R)/\PO_d(\R)$, 
normalized to have $L^2$-norm $1$ \wrt the Riemannian volume $d\vol$.
 
Then the measures $\abs{\psi_i}^2 d\vol$
converge weakly to the Haar measure, \ie for any $f\in C(Y)$,
\begin{equation} \label{weakconvergence}
\lim_{i \rightarrow \infty} \int_{Y} \abs{\psi_i}^2 f  d\vol =
\int_{Y} f d\vol.\end{equation} 
\end{thm} 
 
In words, Theorem \ref{thm:intro} asserts that the eigenfunctions
$\psi_i$ become \emph{equidistributed} -- that they do not
cluster too much on the manifold $Y$. 
 
This theorem is a contribution to the study of the
``Arithmetic Quantum Unique Ergodicity'' problem.
A detailed introduction to this problem may be found in our
paper \cite{SilbermanVenkatesh:SQUE_Lift}.
While it is hard to dispute that the spaces $Y$ are far too special
to provide a reasonable model for the physical problem of
``quantum chaos,'' both the statement \eqref{weakconvergence}
and the techniques we use to prove it seem to the authors
to be of interest because of the scarcity of results concerning
analysis of higher rank automorphic forms.  In particular, we believe
our techniques will find applications to analytic problems beside QUE.

Our strategy follows that of Lindenstrauss in his proof of the
arithmetic QUE conjecture for quotients of the hyperbolic plane
(the case $d=2$ of the Theorem above)
and has three conceptually distinct steps.
Let $A \subset \PGL_d(\R)$ be the subgroup of diagonal matrices. 

\begin{enumerate}
\item \label{microlocal} {\em Microlocal lift:}
notation as above, any weak limit (as $i \rightarrow \infty)$
of $\abs{\psi_i}^2 d\vol$ may be lifted to an $A$-invariant
measure $\sigma_{\infty}$ on $X := \Gamma \backslash \PGL_d(\R)$,
in a way compatible with the Hecke correspondence.

\item \label{tubes} {\em Mass of small tubes:}
If $\sigma_{\infty}$ is as in (\ref{microlocal}), then the
$\sigma_\infty$-mass of an $\epsilon$-ball in
$\Gamma \backslash \PGL_d(\R)$ is
$\ll \epsilon^{d-1 + \delta}$, for some $\delta > 0$; note that
the bound $\ll \epsilon^{d-1 }$ is trivial from the $A$-invariance.
\footnote{This asserts, then, that $\sigma_{\infty}$ has some
  ``thickness'' transverse to the $A$-direction; for instance,
  it immediately implies that the dimension of the support of
  $\sigma_{\infty}$ is strictly larger than $d-1$.}

\item  \label{measurerigidity} {\em Measure rigidity:}
Any $A$-invariant measure satisfying the auxiliary condition prescribed
by (\ref{tubes}) must necessarily be a convex combination of
algebraic measures.  In our setting
($\Gamma$ associated to a division algebra of prime degree)
this means it must be Haar measure.
\end{enumerate}

\begin{center}
* * *
\end{center}

In the context of Lindenstrauss' proof, the analogues of steps
\ref{tubes} and \ref{measurerigidity} are due,
respectively, to Bourgain--Lindenstrauss \cite{LindenstraussBourgain:SL2_Ent}
and Lindenstrauss \cite{Lindenstrauss:SL2_QUE}.  The analouge of
step \ref{microlocal} is due to \cite{Lindenstrauss:HH_QUE} (based
on constructions of Schnirel'man \cite{Schnirelman:Avg_QUE},
Zelditch \cite {Zelditch:SL2_Lift_QE,Zelditch:SL2_Lift_A_inv},
Colin de Verdi\'ere \cite{CdV:Avg_QUE} 
and Wolpert \cite{Wolpert:SL2_Lift_Fejer}).

We shall concern ourselves with the higher rank case
($d>2$), where step \ref{microlocal} -- under a nondegeneracy condition --
has been established by the authors in \cite{SilbermanVenkatesh:SQUE_Lift},
while step \ref{measurerigidity} was established by
Einsiedler-Katok-Lindenstrauss in \cite{EKL:SLn_Rigid}.

The contribution of the present paper is then the establishment of
step \ref{tubes}. 

\subsection{Bounding mass of tubes -- vague discussion}
As was discussed in the previous Section,
the main point of the present paper is to prove upper
bounds for the mass of eigenfunctions in small tubes. 

A \emph{correspondence} on a manifold $X$, for our purposes,
will be a subset $S \subset X \times X$
such that both projections are topological coverings. Such a correspondence
induces an endomorphism of $L^2(X)$: pull back to $Y$ and push forward to $X$.
We also think of a correspondence as a ``multi-valued'' or ``set-valued''
function $h_S$ from $X$ to $X$.  In the latter view a correspondence
induces a natural convolution action on functions on $X$,
given by $(f*h_S)(x) = \sum_{y\in h_{S}(x)} f(y)$.

Two correspondences can be composed in a natural way and resulting
algebra is, in general, non-commutative.  However, the manifolds
of interest to us ($X = \Gamma \backslash \G$ with $\Gamma$
an arithmetic lattice in the semisimple Lie group $G$)
come equipped with a large algebra of \emph{commuting} correspondences,
the Hecke algebra $\mathcal{H}$, which acts on $L^2(X)$ by normal operators.
We will be interested in
possible concentration of simultaneous eigenfunctions of the Hecke algebra.

As a concrete example, for $X = \PGL_d(\Z) \backslash \PGL_d(\R)$
the Hecke correspondences are induced by \emph{left} multiplication
with $\PGL_d(\Q)$: given $\gamma\in\PGL_d(\Q)$ and a coset $x\in X$,
we consider the set of products $\gamma g$ as $g$ varies over
representatives in $\PGL_d(\R)$ for $x$.  It turns out that these
products generate a finite set of cosets $h_\gamma(x)\subset X$.
It is easy to check that the adjoint of $h_\gamma$ is $h_{\gamma^{-1}}$,
but the commutativity of the Hecke algebra is more subtle.
An important feature of the Hecke correspondences on $X$ is their
equivariance \wrt the action of $G = \PGL_d(\R)$ on $X$ on the right.

Returning to the general $X := \Gamma \backslash G$,
let $\tube(\epsilon)$ be a small subset
of $G$, with its size in certain directions on the order of $\epsilon$
(for us $\tube(\epsilon)$ will be a tube of width $\epsilon$ around
a compact piece of a Levi subgroup of $G$).
Our goal will be to prove a statement of the following type, for some
fixed $\eta > 0$ depending only on $G$:
\begin{equation} \label{desideratum}
\mbox{ Each $\mathcal{H}$-eigenfunction $\psi\in L^{2}(X)$
satisfies $\mu_{\psi}(x \tube(\epsilon))\ll \epsilon^{\eta}$.}
\end{equation}
Here $\mu_{\psi} := \abs{\psi}^2 d\vol$ is the product of the
$\PGL_d(\R)$-invariant measure and the function $\abs{\psi}^2$,
normalized to be a probability measure.  \eqref{desideratum} asserts
that the eigenfunction $\psi$ cannot concentrate on a small tube.
This is proven, in the cases of interest for this paper,
in Theorem \ref{thm:tubemass}. 

We will sketch here our approach to the proof.  A basic form of the idea
appeared in the paper \cite{RudnickSarnak:Conj_QUE} of Rudnick and Sarnak.
If  $\psi$ is an eigenfunction of a correspondence
$h \in \Hecke$, and $\psi$ were large at some point $x$,
it also tends to be quite large at points belonging to the orbit
$h.x$.  We can thereby ``disperse'' the local question of bounding
the mass of a small tube to a global question about the size of $\psi$
throughout the manifold.

Say that the image of the point $x$ under $h \in \Hecke$ is the
collection of $N$ points $h.x = \{x_i\}$.  Equivariance implies that the
image of the tube $x \tube(\epsilon)$ under $h$ is the collection
of tubes $\left\{x_i \tube(\epsilon)\right\}$.
For any $t \in \tube(\epsilon)$, we have, then
$$ \lambda_h \psi(xt) =  \sum_{i=1}^N \psi(x_i t), $$
where $\lambda_h$ is so that $h . \psi = \lambda_h \psi$. 

Squaring, applying Cauchy--Schwarz and integrating over
$t \in \tube(\epsilon)$ gives:
\begin{eqnarray}\label{heckecontrol}
\mu_{\psi}(x \tube(\epsilon)) & \leq & \nonumber
\frac{N}{\abs{\lambda_h}^2}  \sum_{i=1}^N \mu_{\psi} (x_i  \tube(\epsilon)) \\
& \leq & \frac{N}{\abs{\lambda_h}^2} \max_i \#
\left\{j \vert x_i\tube(\epsilon)\cap x_j\tube(\epsilon) \neq \emptyset\right\}.
\end{eqnarray}

If the tubes $x_i \tube(\epsilon)$ are \emph{disjoint}\footnote{
  It suffices for the number of tubes intersecting a given one to
  be uniformly bounded independenly of $\epsilon$.}
and, furthermore, $|\lambda_h|$ is ``large'' (\wrt $N$),
\eqref{heckecontrol} yields a good upper bound for
$\mu_{\psi}(x \tube(\epsilon))$.

The issue of choosing $h$ so that $\lambda_h$ is ``large''
turns out to be relatively minor. The solution is given in the appendix.

The more serious problem is that the tubes $x_i \tube(\epsilon)$
might not be disjoint.  It must be emphasized that this issue is
not a technical artifact of the proof but related fundamentally
to the analytic properties of eigenfunctions on arithmetic locally
symmetric spaces: ``returns'' of the Hecke correspondence influece
the sizes of eigenfunctions.  For an instance of this phenomenon see
the Rudnick--Sarnak example of a sequence of eigenfunctions on a
hyperbolic $3$-manifold with large $L^{\infty}$-norms and the more recent
work of Milicevic \cite{Milicevic:LargeValuesH2,Milicevic:LargeValuesH3}.

Our approach to this difficulty is as follows:  we prove a variant
of \eqref{heckecontrol} where the ``worst-case'' intersection number
(i.e.\ $\max_j$) is replaced by an \emph{average} intersection number
(an average over $j$). This variant is presented in Lemma \ref{cover}.
The rest of the paper is then devoted to giving upper bounds for this
average intersection number, which turns out to be much easier than
bounding the worst-case intersection number. 

\begin{rem}
Let us contrast this approach to prior work.
An alternate idea would be to choose a subset of translates $\{x_i\}$
for which we can prove an analogue of \eqref{heckecontrol} and such that
the tubes $x_i \tube(\epsilon)$ are disjoint.  Versions of this were
used in the prior work of Rudnick--Sarnak and Bourgain--Lindenstrauss,
with the quantitative version of Bourgain--Lindenstrauss requiring a
sieving argument to find non-intersecting correspondences.
Our original proof was based on a further refinement of this technique,
which avoided sieves entirely by using some geometry of buildings.
A presentation of that proof may be found in the PhD thesis of
the first author, \cite{Silberman:Thesis}.
The technique of this paper seems to us to be yet more streamlined. 
\end{rem}

\begin{rem}
In order to ``disperse'' the eigenfunction we require the use of Hecke
operators at many primes.  The recent work \cite{BrooksLindenstrauss:OneHecke}
shows that in the case of hyperbolic surfaces, it suffices to use the Hecke
operators at a single place.  Generalizing that result to higher rank would
be an interesting problem.
\end{rem}

\subsection{Spectrum of quotients. Significance of division algebras.}

More generally, the second technique can be interpreted as an implementation
of the following philosophy, related to the work of Burger and Sarnak:

\emph{The analytic behavior of Hecke eigenfunctions on $\Gamma \backslash G$
along orbits of a subgroup $H \subset G$
is controlled by  the spectrum of quotients $L^2(G_p/H_p)$.} 

Here $G_p$ is the $p$-adic group corresponding to $G$, and $H_p \subset G_p$
is a $p$-adic Lie subgroup ``corresponding'' to $H$ in a suitable sense.
In the main situation of this paper, $G_p = \PGL_d(\Q_p)$
for almost all $p$, $H$ will be a real Levi subgroup, and $H_p$ will be a torus. 

In this context, the possibilities for the subgroup $H_p$ that can occur
are closely related to the $\Q$-structure of the group underlying $G$.
In general, the fewer $\Q$-subgroups $\G$ has, the fewer the
possibilities for $H_p$.  For this reason we can only reach Theorem
\ref{thm:intro} for quotients $\Gamma \backslash G$ arising from division
algebras of prime rank: the corresponding $\Q$-groups have
very few subgroups. As one passes to general $\Gamma \backslash G$,
the possibilities for $H_p$ become wilder, and eventually the methods
of this paper do not seem to give much information. 

\subsection{Organization of this paper}
In Section \ref{sec:notation} we describe our setup in the general
setting of algebraic groups.  Further notation regarding our special
case of division algebras is discussed in Section \ref{extranotations}.

Section \ref{sec:tubemass} contains the derivation of our first technical
result, Lemma \ref{cover}, giving a bound for the integral on a small set
of the squared modulus of an eigenfunction of an equivariant correspondence.
This is a version of \eqref{heckecontrol} which can be used for non-disjoint
translates.  The bound we obtain depends on the average multiplicity of
intersection among the translates of the tube as well as on covering properties
of the tubes (easily understood in natural applications).

In Section \ref{sec:dioph} we define the kind of tubes we shall be
interested in and study the intersection patterns of their translates
by elements of the Hecke algebra.  We give two treatments of the analysis,
one that is applicable to general $\R$-split groups, and another,
more concrete, specific to division algebras of prime degree.
In both cases we use a diophantine argument to
show that under suitable hypotheses the intersection pattern is
controlled by a torus in the underlying $\Q$-algebraic group. 

Section \ref{sec:tubemass2} then obtains the desired power-law decay
of the mass of small tubes.  The considerations of this section
are again fairly general.

Finally in Section \ref{sec:results} we recall our previous result
(``step \ref{microlocal}'' of the strategy) and prove our main Theorems.

\subsection{Acknowledgements}
This paper owes a tremendous debt both to Peter Sarnak and Elon Lindenstrauss.
It was Sarnak's realization, developed throughout the 1990s, that the
quantum unique ergodicity problem on arithmetic quotients was a question that
had interesting structure and interesting links to the theory of $L$-functions;
it was Lindenstrauss' paper \cite{Lindenstrauss:SL2_QUE} which introduced
ergodic-theoretic methods in a decisive way.
Peter and Elon have both given us many ideas and comments
over the course of this work, and it is a pleasure to thank them.

The second author was supported by a Clay Research Fellowship,
and NSF Grant DMS--0245606.

Both authors were partly supported by NSF Grant DMS--0111298;
they would also like to thank the Institute for Advanced Study for providing
superb working conditions.

\section{Notation} \label{sec:notation}
We shall specify here the ``general'' notation to be used throughout the paper. 
Later sections (after Section \ref{sec:tubemass}) will use, in addition to
these notations, certain further setup about division algebras. This will
be explained in \S \ref{extranotations}.

Let $\G$ be a semisimple group over $\Q$.
We choose an embedding $\rho: \G \rightarrow \SL_n$. 
Let $G = \mathbf{G}(\R)$, 
$G = NAK$ a Cartan decomposition.  Set $\G(\Z_p) = \rho^{-1}(\SL_n(\Z_p))$,
and let $\Kf \subset \GAf$ be an open compact subgroup contained in
$\prod_{p} \G(\Z_p)$.
We set $X = \GQ \backslash \GA/\Kf, Y = X/K$.
We let $d\vol$ be the natural probability measures on either $X$ or $Y$:
in both cases, the projection of the $\GA$-invariant probability
measure on $\GQ \backslash \GA$.  We shall sometimes denote $d\vol_X$ by
$dx$ and $d\vol_Y$ by $dy$.  We also normalize the Haar measure on $K$ to
be a probability measure.

For any subset $B \subset G$ or $B \subset \G(\adele)$,
we denote by $\overline{B}$ the image of $B$ in $X$ under the
natural projections $G \rightarrow X$ and $G(\adele) \rightarrow X$. 

We say that a prime $p$ is \emph{good} if $\G(\Q_p)$ is unramified
and $K_p = \G(\Z_p) \subset \G(\Q_p)$ is a good maximal compact
subgroup which is contained in $\Kf$
(via the natural embedding $\G(\Q_p) \hookrightarrow \GAf$).
Then all but finitely many $p$ are \emph{good}.
If $p$ is not good, we say it is \emph{bad}.

We denote by $\mathcal{H}$ the Hecke algebra
of $\prod_{p \, \textrm{good}} K_p$ bi-invariant functions
on $\prod_{p \, good} \G(\Q_p)$.  It forms a commutative algebra
under convolution.  It acts in a natural way on functions on $X$. 
We can identify an element of $\mathcal{H}$ with a $\Kf$-invariant
function on $\GAf/\Kf$ which is, moreover, supported
on $\Kf \cdot \prod_{p \, good} \G(\Q_p)$.  We shall abbreviate the latter condition to:
\emph{supported at good primes.} 

If $\H \subset \G$ is a semisimple $\Q$-subgroup,
we say a prime $p$ is $\H$-good
if it is good for $\G$ and moreover
 $\H(\Q_p) \cap K_p$ is a maximal compact subgroup of $\H(\Q_p)$.

We fix compact subsets $\Ominf \subset G, \Omega \subset \GA$, and let
$X_1 \subset X$ denote the (compact) projection of $\Omega$ to $X$.

\begin{defn}
We call $\psi\in L^{2}(X)$ a \emph{Hecke eigenfunction} if
it is a joint eigenfunction of the Hecke algebra $\mathcal{H}$.
We set $\mu_{\psi} = c |\psi|^2 d\vol$, where the constant $c$ is chosen so that $\mu_{\psi}$
is a probability measure on $X$. 
\end{defn}

In \S \ref{sec:tubemass}  we will deal with an
abstract $\G$ as above.  In \S \ref{sec:dioph} --- \S \ref{sec:results}
we sometimes specialize to the
case of $\G$ arising from a division algebra of prime rank, $D$.
The extra notation necessary for this specialization
will be set up in \S \ref{extranotations}.

{\bf Notational convention:} We shall allow the implicit constants in the
notation $\ll$ and $O( \cdot)$ to depend on $\G$ and the data
$\rho, \Kf, \Omega, \Omega_{\infty}$
\emph{without explicit indication}.
In other words, the notation $A \ll B$ means that there exists a
constant $c$, which may depend on all the data specified in this
section, so that $A \leq c B$.  Later we introduce further
data concerning division algebras, and,
as we specify there, we shall allow implicit constants after that
point to depend on these extra data also.

\section{Bounds on the mass of tubes} \label{sec:tubemass}

\subsection{A covering argument} \label{openingsalvo}

Let $B_0 \subset \Omega_{\infty}$ be an open set
containing the identity.  We set
\begin{equation}\label{balli}B :=  B_0 \cdot B_0^{-1}, B_2 = B \cdot B, B_3 = B \cdot B\cdot B, \ldots\end{equation}

In this section, we shall discuss estimating from above
$\mu_{\psi}(\overline{xB})$ for $x \in G$.
\footnote{In fact, all our estimates really work for $\mu_{\psi}(xB)$ for
  $x \in X$, as the notation is meant to suggest. We prove our results only
  for $x \in G$ simply to keep notation to a minimum. If $G$ acts with a
  single orbit on $X$, as we assume in our final applications anyway, there
  is no difference at all.} 
We will repeatedly use the following fact. If $S \subset G$ is a compact subset
and $x \in X_1$ is arbitrary, the cardinality of the fibers of the map
$S \mapsto X$ defined by $s \mapsto x. s$ is bounded
in a fashion that depends on $S$ but not on $x$. 
This is an immediate consequence of the compactness of $X_1$:
indeed, there exists an open neighbourhood $U$ of the identity in $G$
so that $u \in U \mapsto xu \in X$ is injective for any $x \in X_1$. 
In particular, we may apply this remark when $S = B,  B_2, B_3$.
The size of the relevant fibers will therefore be bounded
by a constant depending on $\Omega$.

\begin{lem} \label{cov1}
There exists a covering of $X$ by translates $x_{\alpha} B$
so that any set $z B$, for $z \in X$, intersects at most
$\frac{\vol B_3}{\vol B_0}$ of the $x_{\alpha} B$.
\end{lem}
\begin{proof} 
Choose a maximal subset $ \{x_{\alpha}\} \subset X$ with the property
that $x_{\alpha} B_0$ are disjoint.

Given $z \in X$, we must have $z B_0 \cap x_{\alpha} B_0 \neq \emptyset$
for some $\alpha$.  This means precisely that $X = \bigcup x_{\alpha} B$.

Next, fix $z \in X$. For any
$\alpha$ so that $x_{\alpha} B \cap z B \neq \emptyset$,
we may choose $\varpi_{\alpha} \in B \cdot B^{-1}$ so that
$x_{\alpha} = z \varpi_{\alpha}$.
Then the sets $z \varpi_{\alpha} B_0 \subset X$ are all disjoint.  

A necessary condition for this is that the sets $\varpi_{\alpha} B_0$,
considered as subsets of $G$, are disjoint.
Since each $\varpi_{\alpha} B_0$ belongs to $B_3$, their number is bounded
by $\frac{\vol B_3}{\vol B_0}$. 
\end{proof}

\begin{lem} \label{cov2}
Let $\nu$ be a probability measure on $X$ and $y_1, \dots, y_r \in X$. 

Then
$$\sum_{i=1}^{r} \nu(y_i B)^{1/2} \ll  \frac{\vol B_3}{\vol B_0} \left( \#\left\{ (i,j): y_i B_2 \cap y_j B_2 \neq \emptyset \right\} \right)^{1/2}$$
\end{lem} 
\begin{proof}

Choose a collection $x_{\alpha}$ as in Lemma \ref{cov1}. Each set $y_i B$
is covered by at most $\frac{\vol B_3} {\vol B_0}$ sets $x_{\alpha}B$. Clearly
$ \nu(y_i B)^{1/2} \leq
 \sum_{\alpha: x_{\alpha} B \cap y_i B \neq \emptyset} \nu(x_{\alpha} B)^{1/2}$, 
 and so
 \begin{multline}
 \left( \sum_i \nu(y_i B)^{1/2}  \right)^2
\leq  \left( \sum_{\alpha} \nu(x_{\alpha} B)^{1/2}  \sum_{i: y_i B \cap x_{\alpha} B \neq \emptyset} 1  \right)^2 \\ \leq 
 \sum_{\alpha} \nu(x_{\alpha} B) \sum_{\alpha}  \left(\sum_{i: y_i B  \cap x_{\alpha} B \neq \emptyset} 1 \right)^2  
\end{multline}

Now, $\sum_{\alpha} \nu(x_{\alpha} B) \leq \frac{\vol B_3}{\vol B_0}$ because each $z \in X$
belongs to at most $\frac{\vol B_3}{\vol B_0}$ of the $x_{\alpha} B$.

Moreover, 
$$ \sum_{\alpha} \left(\sum_{i: y_i B  \cap x_{\alpha} B \neq \emptyset} 1 \right)^2   
= \#\{i,j, \alpha: y_i B \cap x_{\alpha} B \neq \emptyset, 
y_j B \cap x_{\alpha} B \neq \emptyset \}$$
For given $i,j$, the pertinent set of $\alpha$ is nonzero only if 
$y_i B_2 \cap y_j B_2 \neq \emptyset$. If it is nonempty, it has size at most
$\vol(B_3)/\vol(B_0)$. 
\end{proof}

\subsection{General bound}

Let $\psi$ be a Hecke eigenfunction on $X$,
$\mu_{\psi}$ the associated probability measure.

Let $h \in \Hecke$, which we can think of as a function $s \mapsto h_s$
on $\GAf/\Kf$, supported on good primes.
Let $S \subset \GAf/\Kf$ be the support of $h$. 

Because $\psi$ is a Hecke eigenfunction, there is $\Lambda_h \in \C$ so that:
\begin{equation} \label{hecke}
\Lambda_h \psi(x) = \sum_{s \in S} h_s \psi(x. s) \ \ \ (x \in G)
\end{equation}

Note that $x.s \in \GA/\Kf$ and therefore $\psi(x.s)$ makes sense.
The following Lemma bounds $\mu_{\psi}(\overline{x B})$ in terms of the
average mass of certain Hecke translates of $\overline{xB_2}$. 

\begin{lem} \label{meansquare} 
Let $x \in X_1$. Then
$$\mu_{\psi}(\overline{x B}) \ll  \frac{ \left(  \sum_{s \in S}  |h_s| \mu_{\psi}(\overline{x B} . s)^{1/2}
\right)^2.}{|\Lambda_h|^2}$$
\end{lem}

\begin{proof} 
This follows by squaring out equation \eqref{hecke}, integrating over $B$
and applying Cauchy-Schwarz. 
We use the fact that $b \in B \mapsto x.s.b \in X$ has fibers whose cardinality is bounded
in terms of $\Omega_{\infty}$ (see discussion in \S \ref{openingsalvo}). 
\end{proof}

The next Lemma clarifies that the only necessary input to bound $\mu_{\psi}(\overline{xB})$
is an estimate for the average intersection multiplicity of Hecke translates of $\overline{xB}$. 

\begin{lem} \label{cover}
Suppose that $h$ is supported on $S \subset \GAf/\Kf$ and $|h| \leq 1$. 
 
Then for $x \in G$
\begin{equation} \label{inter}
\mu_{\psi}(\barxB) \ll 
\left( \frac{\vol B_3}{\vol B_0} \right)^2 |\Lambda_h|^{-2} 
\left( \#\{s, s' \in S: \overline{xB_2  s} \cap \overline{xB_2  s'} \neq \emptyset\} \right)
\end{equation}
\end{lem}
\begin{proof}
This follows from the prior Lemma and Lemma \ref{cov2} (applied to the set $\{y_i\} = \{x s: s\in S\} \subset X$). 
 \end{proof}

Lemma \ref{cover} is our key technical Lemma.  Setting $N = \# S$ as in
the introduction we may rewrite the right-hand-side as:
$$ 
\left( \frac{\vol B_3}{\vol B_0} \right)^2
\frac{N}{\abs{\Lambda_h}^2}  \cdot
\frac{1}{N}\sum_{s\in S}
 \# \left\{ s' \in S: \overline{x B_2 s} \cap \overline{x B_2 s'} \neq \emptyset \right\}.
$$
This expression should be compared with the right-hand-side of
equation \eqref{heckecontrol}.  The key difference is that the bound
now only depends on an ``average'' intersection number
of Hecke translates (the average over $s$ of the number of $s'$ such that
the translates by $s$ and $s'$ intesect), whereas the bound in
\eqref{heckecontrol} depended on a ``worst case'' intersection number
(the supremum over $s$ of the number of such $s'$).
Our antecedent \cite{LindenstraussBourgain:SL2_Ent}
relied on controlling this latter quantity, which imposed greater
restrictions on the use of Hecke operators.

\section{Diophantine Lemmata}\label{sec:dioph}
Recall that we have fixed a maximal $\R$-split torus $A \subset G$.
Given a nontrivial $a \in A$, we fix a compact neighbourhood of the
identity $C \subset \Z_G(a)$ in the centralizer  $a$ inside $G$
(the choice is immaterial; we will later take $C$ to be small enough).
Now let $B=B(C,\varepsilon)$ be an $\varepsilon$-neighbourhood
of $C$ inside $G$ (``a tube around a piece of a Levi subgroup'').
We intend to bound the $\mu_{\psi}$-mass of sets of the form
$\overline{x B(C,\varepsilon)} \subset X$.

In the previous Setion we saw that such bounds require control
on the intersection pattern of translates of these sets.  This Section
is devoted to two results giving this control.

We first analyze the case where $\G$ is $\R$-split and $a\in A$
is \emph{regular}, that is when $Z_G(a)$ is a maximal torus.
It turns out that the intersection is controlled by a $\Q$-subtorus
of $\G$.  Roughly, we show that all the $\gamma \in \G(\Q)$ that lie
``very close'' to an $\R$-torus (where ``very close'' really means
``very close relative to the denominator of $\gamma$'') must all lie
on a single $\Q$-torus.  Using this we will show that the intersection
pattern of Hecke translates of tubes of this type are
controlled by the torus.
 
Secondly, more control is possible in the simplest case:
when $\G$ arises from a division algebra of prime degree.
In that case $\G$ has very few $\Q$-subgroups and one may remove
the regularity assumption: elements close to any Levi subgroup (a subgroup
of the form $Z_G(a)$) must lie on a $\Q$-torus.

An argument of this type (in the case $\dim_{\Q} D=4$) already appears
in \cite{LindenstraussBourgain:SL2_Ent}; this argument does not suffice
for the higher rank case, however, because \cite{LindenstraussBourgain:SL2_Ent}
uses commutativity of the relevant Levi subgroups in an important
way.  Also, it is somewhat awkward in the case of algebras which are
not division algebras.  However, this approach is more transparent.
Our second argument takes this \emph{ad hoc} approach, and is a
generalization specific to the case of division algebras.

\begin{rem} Here is a toy model of the type of reasoning we use:
Let $\ell$ be a line segment in $\R^2$ of length $1$.  Suppose
$P_i = (x_i, y_i)$, for $1 \leq i \leq 3$, are points in $\R^2$ with
rational coordinates all of which lie within $\varepsilon$ of $\ell$,
and let $M$ be an upper bound for the denominators of all $x_i, y_i$.
Then, if $\varepsilon < \frac{1}{10} M^{-6}$, the $P_i$ are themselves
co-linear.  Indeed, the area of the triangle formed by $P_1, P_2, P_3$
is a rational number with denominator $\leq 2 M^6$.  On the other hand
the area of this triangle is $\leq 2 \varepsilon$, whence the conclusion.
\end{rem}

\begin{rem}
Recall that we have fixed compact subsets
$\Omega \subset \G(\adele)$ and $\Omega_{\infty} \subset G$. 
We will later take $C$ and $\epsilon$ small enough to ensure that
$B(C,\varepsilon) \subset \Ominf$. 
In particular, $C \subset \Ominf$. 
Consequently, in view of the convention discussed
in \S \ref{sec:notation}, we shall not explicitly indicate that
implicit constants in $\ll$ or $O(\dots)$ depend on $C$. 
\end{rem}

\subsection{Denominators on an adelic group}\label{sec:adeldenom}

As in the toy example, in our diophantine anaylsis it is
convenient to use a notion of \emph{denominator}
rather than a notion of height.  We define this notion
progressively for scalars, for elements of $\SL_n(\Af)$,
and then for elements of $\GAf$.  We write $\N$ for the set of
natural numbers, where our denominator function will be valued.

Given $x\in\Q_p$ we let $\denom(x)$ denote the minimal non-negative power
of $p$ such that $\denom(x)\cdot x\in \Z_p$.  In other words,
$p$-adic integers have no denominator while a $p$-adic number of
the form $u/p^e$, $u\in\Z_p^\times$, has denominator $p^e$.
It is easy to check that the function $\denom\colon\Q_p\to\N$
is invariant under translation by $\Z_p$ and is hence
uniformly continuous.  Next, for $x=(x_p)_{p<\infty} \in \Af$ we set
$\denom(x)=\prod_{p<\infty} \denom(x_p)$
(almost all the factors are equal to $1$).
This gives a locally constant function $\denom\colon\Af\to\N$.
The diagonal embedding of $\Q$ in $\Af$ allows us to restrict $\denom$ to $\Q$
and it is easy to check that if $a,b\in\Z$ are relatively prime with $b>0$
then $\denom(\frac{a}{b}) = b$.  In other words, the restriction is the
usual notion of the denominator of a rational number.

For $g=(g_p)\in \SL_n(\Af)$ set $\denom(g)$ to be the least common multiple
of the denominators of the matrix entires.  
For $k_p \in \SL_n(\Z_p)$, or more generally
for $k\in \prod_{p} \SL_n(\Z_p)$ we then have $\denom(k_p) = \denom(k) = 1$,
and in fact $\denom\colon \SL_n(\Af) \to \N$ is left- and right-invariant
by $\prod_{p} \SL_n(\Z_p)$.

Finally, for $g \in \GAf$, we set $\denom(g) = \denom(\rho(g))$.
This function is then bi-$\Kf$-invariant.
By restriction this induces a notion of denominator
for elements $\gamma\in\G(\Q)$, the least common multiple of the denominators
of the matrix elements of $\rho(\gamma)\in\SL_n(\Q)$.

For the infinite place we fix a left-invariant Riemannian
metric $\dist$ on $\SL_n(\R)$.  Pulling the metric tensor back
via $\rho$ induces a left-invariant Riemannian metric $\distG$
on $G$.  Since it maps geodesics on $G$ to curves of the same
length in $\SL_n(\R)$, $\rho$ is a nonexpansive map with respect to
these metrics.

The ``$\varepsilon$ neighbourhoods'' $B(C,\varepsilon)$ are taken with
respect to the metric $\distG$.

\begin{lem}
Let $g,g'\in \GAf$. Then $\denom(gg')\leq \denom(g)\denom(g')$
and $\denom(g^{-1})\leq \denom(g)^{n-1}$.
\end{lem}
\begin{proof}
Replacing $g,g'$ by their images under $\rho$ we may
assume $g,g'\in \SL_n(\Af)$.  For the first claim we use an alternate
characterization of $\denom(g)$: it is the smallest positive integer $m$
for which $m\cdot g\in \GL_n(\prod_p \Z_p)$.  This implies
$$\denom(g)\denom(g')gg'\in \GL_n(\prod_p \Z_p)\, ,$$
and the claim follows.
For the second note that the matrix entries of $g^{-1}$ are polynomials
of degree $n-1$ and integer coefficients in the matrix entries of $g$. 
\end{proof}

\subsection{A Diophatine Lemma for $\R$-split groups}\label{subsec:dl1}
Assume now that $\G$ is $\R$-split, and let $a\in A$ be a \emph{regular}
element. In that case $Z_G(a)=MA$ is a maximal torus in $\G(\R)$,
and $\rho(MA)$ is an $\R$-split torus in $\SL_n(\R)$. 

The following is the basic diophantine result.
\begin{lem}\label{diophantine1} For $c>0$ sufficiently
large (in fact, depending only on $\G, \rho$), and $c'>0$ sufficiently small
(depending on $\G, \rho, A, \Ominf$), for any $g\in\Ominf$
the set of $\gamma\in\GQ$ such that
\begin{equation}\label{defSsplit}
\inf \{\distG(\gamma,t) \mid t\in g(MA\cap\Ominf\Ominf^{-1})g^{-1}\}
  \leq \varepsilon ,{\,\,} \denom(\gamma)\leq M
\end{equation}
is contained in a $\Q$-subtorus $\T\subset \G$,
provided that
\begin{equation}
\varepsilon M^c \leq c'
\end{equation}
\end{lem}

Let $S$ be the set of $\gamma$ defined in equation
\eqref{defSsplit}.  We establish two preparatory results before the
Lemma itself.  The first is based on the analysis of the case
of division algebras (see Lemma \ref{diophantine2}) where the idea is more
fully exploited.

\begin{lem}[enough to check subsets of fixed size]\label{lem:gen-torus}
A set $S \subset \GQ$ is contained in a $\Q$-torus iff the same holds
for every subset $S'$ of $S$ of size $n+1$.
\end{lem}
\begin{proof}
For each $S'$ of size $n+1$ let $A_{S'}\subset M_n(\C)$ be the
$\C$-algebra generated in $M_n(\C)$ by the image set $\rho(S)$,
and choose $S'$ such that $A'=A_{S'_0}$ has maximal dimension
(as a vector space over $\C$).

Since $\rho(S')$ is contained in a subtorus of $\SL_n$, the subalgebra
$A'$ is conjugate to a subalgebra of the algebra of diagonal matrices.
In particular, it has dimension at most $n$.  In a proper containment
of $\C$-algebras, the two algebras have different dimensions.
Thus there exists a subset $S_0 \subset S'$ of size at most $n$ such
that $A_{S_0}=A'$.  

Further, for any $\gamma\in S$, $A_{S_0\cup\{\rho(\gamma)\}}$ contains
$A'$ and has at most the same dimension.  It follows that they are equal,
and hence that $A_S = A'$.  Let $\T$ be the intersection with
$\G$ of the $\Q$-subalgebra of $M_n(\Q)$ generated by $\rho(S)$.  This is
a $\Q$-subgroup of $\G$ containing $S$. This subgroup is contained in
$A'$.  That $\T$ is a torus follows from the fact that its
$C$-points are contained in $A'$, hence conjugate to a set of
diagonal matrices.
\end{proof}

\begin{lem}[checking can be done algebraically]
There exists (finite) sets of polynomials
$P_{ij}, Q_{ik} \in \Z[x_1,\ldots,x_{(n+1)n^2}]$
(to be thought of functions of $n+1$ matrices of size $n$),
such that $\{\gamma_l\}_{l=1}^{n+1}\subset\GQ$
all lie on a $\Q$-torus iff for some $i$, $P_{ij}(\rho(\vec{\gamma}))=0$
and $Q_{ik}(\rho(\vec{\gamma})) \neq 0$ hold for all $j,k$.
\end{lem}
\begin{proof}
The analysis of the previous Lemma shows that the desired statement
about $\rho(\vec{\gamma})$ is equivalent to the following formula in
the language of fields:
\begin{center}
{\parbox{4in}
  {There exist matrices $D,A\in\GL_n(\C)$ such that $D$ is diagonal
   with distinct entries and such that each $\rho(\gamma_k)$ commutes
   with $ADA^{-1}$.}
}
\end{center}
By Chevalley's Theorem on elimination of quantifiers for the theory
of algebraically closed field, this formula is equivalent (over $\C$)
to a one without quantifiers, which we may assume to be in the normal
form
$$ \bigvee_i \left( \bigwedge_j \left(P_{ij}(\rho(\vec{\gamma}))=0\right)
\bigwedge_k \left(Q_{ij}(\rho(\vec{\gamma}))\neq 0\right) \right)$$

We reiterate the point that the $\gamma_k$
generate a $\Q$-torus iff they generate a torus over $\C$, and that
the language of algebraically closed fields has no names for field
elements other than $0,1$ so that atomic formulas in it are equalities
and inequalities of rational polynomials, which without loss of
generality may be assumed integral.
\end{proof}

\paragraph*{Proof of the Diophantine Lemma \ref{diophantine1}.} 
We note first that the polynomials $P_{ij}, Q_{ik}$ were only constructed
from $\G$ and $\rho$.  Taking $c'<1$ we may assume that $\varepsilon<1$,
so that all the $\gamma$ are drawn from the $1$-neighbourhood of the
compact set $\Ominf(MA\cap\Ominf\Ominf^{-1})\Ominf^{-1}$.
Let $L$ be a bound for the
Lipschitz constant of the smooth functions $P_{ij}\circ\rho, Q_{ik}\circ\rho$
in this domain.  Again by compactness, there exists $b>0$ such that for
each $g\in\Ominf$ and each $\vec{t}\in(MA\cap\Ominf)^{n+1}$ for the $i$
for which our system of constraints holds for $g\vec{t}g^{-1}$,
and each $k$, $\left|Q_{ik}(g\vec{t}g^{-1})\right| > b$.

Fixing $g\in\Ominf$, now let $\vec{\gamma}$ from $S^{n+1}$,
and let $\vec{t}\in(MA\cap\Ominf)^{n+1}$ so that $\gamma_l$ is
$\varepsilon$-close to $gt_lg^{-1}$.  By the analysis above, there
exists $i$ such that $P_{ij}(\rho(\vec{t}))=0$ while $Q_{ik}(\rho(\vec{t}))$
are at least $b$ in magnitude.  It follows that the magnitude
$Q_{ik}(\rho(\vec{\gamma}))$ is at least $b-L\varepsilon$, which is positive
as long as we ensuer $c'<b/L$.  It also follows that the magnitude of
$P_{ij}(\rho(\vec{\gamma}))$ is at most $L\varepsilon$.
Let $c$ bound the total degree each $P_{ij}$.  Then the denominator
of each rational number $P_{ij}(\rho(\vec{\gamma}))$ is at most $M^c$.
If $c'<1/L$ this ensures that these rational numbers vanish.
\qed

\subsection{The intersection pattern of translates of tubes around Levi subgroups}\label{subsec:ctrltubes}

\begin{prop} \label{propertystar1}
Let $\G$ be $\R$-split.  Fix a relatively compact open neighbourhood
of the identity $C\subset A$ and assume $C\subset\Ominf$.
There are $c_2, c_4>0$, depending only on the isomorphism class
of $\G$ and $c_1, c_3 = O_{C}(1)$,
so that

For any $x = (x_{\infty}, \xf)  \in \Omega$ and any
$0 < \varepsilon < c_3$ there 
exists a $\Q$-subtorus $\T \subset \G$ so that: 
\begin{enumerate}
\item 
If $s,s' \in \GAf$
both have denominator $\leq c_1 \varepsilon^{-c_2}$ and are so that
$$\overline{ x B(C, \varepsilon) s} \cap \overline{ x B(C, \varepsilon) s'} \neq \emptyset \mbox{ in } \G(\Q) \backslash \GA/\Kf$$
then there exists
$\gamma \in \T(\Q) \subset \GQ$ with
\begin{equation} \label{intersect1} \gamma x B(C, \varepsilon) s  \cap x B(C, \varepsilon)  s' \neq \emptyset \mbox{ in } \GA/\Kf.\end{equation}

\item
 There at most $O(1-c_4\log\varepsilon)$ primes which are $\T$-bad.

\end{enumerate}
\end{prop} 

\begin{proof}
Let  $s,s' \in \GAf$ satisfy the intersection condition of item (1).
By assumption there is -- after replacing $s,s'$ by suitable elements
of $s \Kf$ and $s'\Kf$ respectively -- an element $\gamma \in \G(\Q)$
so that $\gamma \xf s = \xf s'$ (equality in $\GAf$) and, moreover,
$\gamma \in x_{\infty} B(C,\varepsilon) B(C, \varepsilon)^{-1} x_{\infty}^{-1}$
(equality in $\G(\R)$).  

The former equality implies, in particular, that
$\denom(\gamma) \ll c_1' \varepsilon^{-c_2'}$, where $c_i'$ depends on $c_i$,
and $c_i' \rightarrow 0$ as $c_i \rightarrow 0$ for $i=1,2$. 

Since conjugation by elements of a compact set is a map of bounded Lipchitz
constant \wrt the metric $\distG$, the latter inclusion shows that
$\gamma$ lies within $L\epsilon$ of
$x_\infty (MA\cap\Ominf\Ominf{^-1})x_\infty^{-1}$, where $L$ depends only
on $\G$ and $\Ominf$.

Let $R$ be the set of such $\gamma$, let $\T\subset\G$ be the closed subgroup
they generate, and let $E$ be the subalgebra of $M_n(\Q)$ generated by
$\rho(R)$.  If $c_1', c_2'$ are
sufficiently small -- this occurs, in particular, if $c_1, c_2$ are
sufficiently small -- then Lemma \ref{diophantine1} shows that $\T$ is a torus,
and $E$ is its linear span, a semisimple abelian subalgebra of $M_n(\Q)$.
Analyzing the reasoning shows that the exponent $c_2$ may be taken to depend
only on the isomorphism class of $\G$, whereas $c_1 = O(1)$.
This proves \eqref{intersect1}.

For a $\G$-good prime $p$ to be $\T$-good, it suffices that
$E_p \cap M_n(\Z_p)$ is a maximal compact subring of
$E_p = E\otimes\Q_p \subset M_n(\Q_p)$.  For this it suffices to have
generators $\{\gamma'_i\}_{i=1}^{n}\subset M_n(\Z)$ for $E$ as a $\Q$-algebra
such that $\Z_p[\gamma'_1,\ldots,\gamma'_d] = E_p \cap M_n(\Z_p)$.  That
will happen as long as $p$ does not divide the discriminant of the
characteristic polynomial of each $\gamma'_i$.

By the proof of Lemma \ref{lem:gen-torus}, there exist
$\{\gamma_i\}_{i=1}^{n}\subset R$ which generate $E$, and let
$\gamma'_i = \denom(\gamma_i) \cdot \gamma_i$.  Then $\gamma'_i \in M_n(\Z)$,
still generate $E$ as a $\Q$-algebra.  Next, since
$\gamma_i \in \Ominf\Ominf\Ominf^{-1}\Ominf^{-1}$ and as $\rho$ is continuous,
the matrix entries of $\gamma'_i$ are $O(\epsilon^{-c_2'})$.  Finally,
the discriminant of $\gamma'_i$ is a polynomial in the coefficients
of its characteristic polynomial, themselves polynomials in the matrix entries
of $\gamma'_i$.

It follows that the set of $\G$-good but $\T$-bad primes is contained
in the set of prime divisors of an integer bounded by $O(\epsilon^{-O(1)})$.
\end{proof}

\paragraph{Generalizations}
Proposition \ref{propertystar1} is a statement of the following type:
\begin{quote}
Let $H\subset G$ be a closed subgroup.  Then given $x\in \Ominf$
and a tubular neighbourhood $B(C,\varepsilon)$ of a piece $C\subset H$,
there exists a $\Q$-subgroup $\T \subset \G$ such that
intersection of Hecke translates of small denominator of $xB(C,\varepsilon)$
are controlled by $\T$, in the sense that if $xBs$ and $xBs'$ intersect
in $X$, there exists $\gamma \in \T(\Q)$ such that $\gamma xBs\Kf = xBs'\Kf$
holds in $\GA/\Kf$.
\end{quote}

We have established this for $\G$ which is $\Q$-anisotropic and $\R$-split
and $H$ a maximal $\R$-split torus.  Specializing further to the case of
$\G$ associated to a division algebra of prime degree, we establish a
result of this type for any Levi subgroup $H\subset G$
(that is, for a subgroup of the form $H=Z_G(a)$, $a\in A$).
It turns out that the subgroup $\T$ remains a torus.

In general one would expect that points lying near pieces of orbits
of Levi subgroups defined over $\R$ to lie on some like an orbit
of a Levi subgroup defined over $\Q$.  This is not quite correct, but
precise versions of this intuition can be proved; see the work 
\cite[\S4]{Marshall:KsmallLinftyBds}.

\subsection{Extra notations for the case of division algebras}
\label{extranotations}

Let $D$ be a division algebra over $\Q$ of prime degree $d$,
and fix a lattice $D_{\Z} \subset D$
(\ie a free $\Z$-submodule of maximal rank)
and a Euclidean norm $\|\cdot\|$ on $D \otimes_{\Q} \R$.
In other words, we have chosen a norm on $D \otimes \Q_v$ for all $v$,
finite or infinite.  Since the only central division algebras over $\R$
are $\R$ itself and Hamilton's quaternions, assuming $d\geq 3$ ensures
that $D\otimes\R$ is the full matrix algebra.

We will consider the case where $\G$ is the projectivized group of units
(=invertible elements) of $D$,
so (for $d\geq 3$) $G = \G(\R)$ is isomorphic to $\PGL_d(\R)$.
The Lie algebra of $G$ is identified with a quotient
of $D \otimes \R$; as such, the norm on $D \otimes \R$ gives rise to a norm
on the Lie algebra of $G$ and thus to a left-invariant Riemannian metric on $G$.

We fix extra data ($\rho, \Kf, \Omega, \Ominf$)
for the group $\G$, as discussed in \S \ref{sec:notation}.
In the rest of this paper, when discussing this case the implicit constants
in the notations $\ll$ and $O(\cdot)$ will be allowed to depend on
$D, D_{\Z}$, the norm $\|\cdot\|$ and this extra data,
without explicitly indicating this. 

It should be noted that we do not assume that
$D_{\Z} . D_{\Z} \subset D_{\Z}$;
on the other hand, clearly there is an integer $K  = O(1) $ so that
$D_{\Z} . D_{\Z} \subset K^{-1} D_{\Z}$.

\subsection{A diophantine lemma for $\Q$-algebras}\label{subsec:dl2}

In this section only we shall use an additional notion of denominator,
special to the case of $\Q$-algebras.  We fix a central simple
$\Q$-algebra $D$ of dimension $d^2$ and a $\Z$-lattice $D_\Z \subset \Q$.
Given $x \in D$, we set
\begin{equation}\label{denomdef}
\denp(x) := \inf\{m \in \N: \ mx \in D_{\Z}\}.
\end{equation} 

Recall that for $\gamma \in \G(\Q) = D^{\times}/\Q^{\times}$,
we also have the denominator
$\denom(\gamma)$ defined in Section \ref{sec:adeldenom}.
We first clarify the relation between the two notions.

\begin{lem} \label{denominators}
Let $\gamma \in \G(\Q) = D^{\times}/\Q^{\times}$ satisfy $\denom(\gamma) \leq M$
and belong to a compact subset $\cmpct \subset \G(\R)$.  Then there exists
$\alpha \in D^{\times}$ lifting $\gamma$ so that:
$$\denp(\alpha) \ll_{\cmpct} M^{c}, \|\alpha^{-1}\| ,
  \|\alpha\| \ll_{\cmpct} 1$$
where $c$ is a constant depending only on the isomorphism class of $\G$. 
\end{lem} 

\begin{proof} 
In fact, let $\Gtwid$ be the algebraic group corresponding to $D^{\times}$, \ie
$\Gtwid(R) = (D \otimes_{\Q} R)^{\times}$ if $R$ is a ring containing $\Q$.

Then $\Gtwid$ and $\G$ are affine algebraic groups. We first show that
the map $\Gtwid \rightarrow \G$
\emph{admits an algebraic section over a Zariski-open set $U \subset \G$}.

Let $\G^{(1)}$ denote the group of elements of norm $1$ in $D^{\times}$.
It is a (geometrically) irreducible variety, because $\mathrm{SL}_d$
is an irreducible variety. The map $\G^{(1)} \rightarrow \Gtwid$ is a
covering map (i.e. {\'e}tale)
and its kernel is the group of $d$th roots of unity.
Let $E$ be the function field of $\G$, considered as $\Q$-variety. 
The generic point in $\eta \in \G(E)$ does not lift to a point of
$\G^{(1)}(E)$, but it does at least to lift to a point of
$\tilde{\eta} \in  \G^{(1)}(\tilde{E})$ for some finite
extension $\tilde{E}/E$, which we may assume to be Galois
and to contain the $d$th roots of unity.
Then $\sigma \mapsto \tilde{\eta}^{\sigma} / \tilde{\eta}$ defines
a $1$-cocycle of $\Gal(\tilde{E}/E)$ valued in the group of $d$th
roots of unity. By Hilbert's theorem 90, there exists
$\tilde{e} \in \tilde{E}$ so that this cocycle is
$\sigma \mapsto \tilde{e}^{\sigma}/\tilde{e}$. 
Adjusting $\tilde{\eta}$ by $\tilde{e}$ gives a $\tilde{E}$-valued point
of $\Gtwid$, which is invariant under $\Gal(\tilde{E}/E)$ and therefore
is indeed an $E$-valued point of $\Gtwid$. This gives the desired section.

One may, by translating $U$, find a finite collection of open sets
$U_1, \dots, U_h$ which cover $\G$, and so that $\Gtwid \rightarrow \G$
admits a section $\theta_j: U_j \rightarrow \Gtwid$ over each $U_j$.  
 
It follows from this that there exists $\alpha \in D^{\times}$
lifting $\gamma$ so that $$\denp(\alpha), \denp(\alpha^{-1})  \ll M^{c}$$
where $c$ is a constant depending only on the choice of
sets $U_j$ and the sections, i.e. only the isomorphism class of $\G$. 

From this bound, it follows in particular that
$M^{-c} \ll_\cmpct \|\alpha\| \ll_{\cmpct} M^{c}$.
The lower bound is clear;
for the upper bound, we use the fact that $\alpha$ projects
to the compact subset $\cmpct \subset \G(\R)$. 

Let $p/q$ be a rational number satisfying $\|\alpha\| < p/q < 2 \|\alpha\|$.
We may choose $p,q$ so that $\max(p,q) \ll M^c$.
Replacing $\alpha$ by $q \alpha/p$, 
we obtain a representative $\alpha$ for $\gamma$ that satisfies:
$$\denp(\alpha) \ll_{\cmpct} M^{2c}, \|\alpha\| \asymp_{\cmpct} 1$$
We increase $c$ as necessary.

Finally, the bound for $\|\alpha^{-1}\|$ follows from the bound
for $\|\alpha\|$ together with the fact that $\alpha$ projects
to the compact set $\cmpct \subset \G(\R)$. 
\end{proof}

The following should be compared with Lemma \ref{diophantine1}.

\begin{lem} \label{diophantine2}
Let $S \subset D \otimes \R$ be a proper $\R$-subalgebra. 

For $c >0$ sufficiently large (in fact, depending only on $d$) and
for $c'>0$ sufficiently small (in fact, depending only on $D, D_{\Z},
\|\cdot\|$), the set of $x \in D$ satisfying 
\begin{equation} \label{xsetdef}
\|x\| \leq R, \inf_{s \in S} \|x-s\| \leq \varepsilon,
\denp(x) \leq M \end{equation}
is contained in a proper subalgebra $F \subset D$ as long as
\begin{equation} \label{cond} \varepsilon R^{c} M^{c} < c'\end{equation}
\end{lem} 

In other words: points of $D$ near a proper subalgebra of
$D \otimes \R$ lie on a proper $\Q$-subalgebra of $D$.
This proof will not use the fact that $D$ is a {\em division} algebra,
nor the fact that it is of prime rank.

\begin{proof}
We use here $f_1(d), f_2(d), \dots$
to denote positive quantities that depend on the rank $d$ alone. 

Let $s = \dim(S)+1$. Then there is a polynomial function
$G: D^{s} \rightarrow \Q$, with integral coefficients with respect
to $D_{\Z}$, so that $G(\alpha_1, \dots, \alpha_s ) = 0$
exactly when $\alpha_1, \dots, \alpha_s$ span a linear space
of dimension $\leq s-1$.  For example one may use the sum of the squares
of the minors of a suitable matrix.
The degree of $G$ is $f_1(d)$ and the size of its coefficients is $O(1)$. 

Take $x_1, \dots, x_s$ belonging to the set defined by \eqref{xsetdef}.
There are $y_1, \dots, y_s \in S$ so that $\|x_i - y_i\| \leq \varepsilon$.
Then $G(x_1, \dots, x_s) \ll R^{f_2(d)} \varepsilon$.
On the other hand, if $G(x_1, \dots, x_s) \neq 0$ then,
because $\denom(x_i) \leq M$, we must have
$G(x_1, \dots, x_s) \gg M^{-f_3(d)}$.
It follows that, if a condition of the type \eqref{cond} holds
for suitable $c,c'$ as stated,
then $x_1, \dots, x_s$ span a $\Q$-linear space of dimension $s-1$. 

Now let $X$ be the $\Q$-algebra spanned by those $x$
satisfying \eqref{xsetdef}. It is clear that $X$ is,
in fact, spanned by monomials in such $x$ of length at
most $\dim_{\Q} D$.  Each such monomial $y$ satisfies
$\|y\| \ll R^{f_4(d)}, \inf_{s\in S} \|y-s\| \ll R^{f_5(d)} \varepsilon,
\denom(y) \gg M^{f_6(d)}$.
It follows that -- increasing $c$ and decreasing $c'$
in \eqref{cond} as necessary -- it follows that the $\Q$-subalgebra
generated by all solutions to \eqref{xsetdef} has dimension $\leq s-1$,
in particular, is a proper subalgebra of $D$. \end{proof}

\begin{prop} \label{propertystar2}
Let $\G$ be the projectivized group of units of a division algebra $D/\Q$
of prime degree $d$.  There are $c_2, c_4> 0$, 
depending only on the isomorphism class of $\G$ and $c_1, c_3 = O_{C}(1)$,
so that

For any $x = (x_{\infty}, \xf)  \in \Omega$ and any
$0 < \varepsilon < 1/2$ there 
exists a subfield $F \subset D$ so that: 
\begin{enumerate}
\item 
If $s,s' \in \GAf$
both have denominator $\leq c_1 \varepsilon^{-c_2}$ and are so that
$$x \overline{B(C, \varepsilon) s} \cap \overline{ x B(C, \varepsilon)  s'} \neq \emptyset \mbox{ in } \G(\Q) \backslash \GA/\Kf$$
then there exists
$\gamma \in F^{\times}/\Q^{\times} \subset \G(\Q)$ with
 \begin{equation} \label{intersect2} \gamma x B(C, \varepsilon) s  \cap x B(C, \varepsilon)  s' \neq \emptyset \mbox{ in } \GA/\Kf.\end{equation}

\item
 $F$ is generated
 by $\alpha \in D^{\times}$, so that $\alpha D_{\Z} + D_{\Z} \alpha  \subset D_{\Z}$, and with $\|\alpha\| \leq  c_3 \varepsilon^{-c_4}$. 

\end{enumerate}
\end{prop} 

\begin{proof}
Let  $s,s' \in \GAf$ satisfy the intersection condition of item (1).
As in the general case before we find -- after replacing $s,s'$ by suitable
elements of $s \Kf$ and $s'\Kf$ respectively -- an element $\gamma \in \G(\Q)$
so that $\gamma \xf s = \xf s'$ (equality in $\GAf$) and, moreover,
$\gamma \in x_{\infty} B(C,\varepsilon) B(C, \varepsilon)^{-1} x_{\infty}^{-1}$
(equality in $\G(\R)$).  

Again, this implies
$\denom(\gamma) \ll c_1' \varepsilon^{-c_2'}$, where $c_i'$ depends on $c_i$,
and $c_i' \rightarrow 0$ as $c_i \rightarrow 0$ for $i=1,2$. 

The latter inclusion shows that $\gamma$ lies in a fixed compact subset
of $\G(\R)$ depending only on $\Omega, C$. 

The element $\gamma$ belongs to $D^{\times}/\Q^{\times}$.
We may choose a representative $\alpha \in D^{\times}$ for $\gamma$ as
in Lemma \ref{denominators}.  In that case $\alpha$ lies in a fixed compact
subset of $(D \otimes \R)^{\times}$ -- depending only on $\Omega, C$ -- and
$\denp(\alpha) \ll  c_1'' \varepsilon^{-c_2''}$,
where (for $i=1,2$) $c_i''$ depends on $c_i$ and $c_i'' \rightarrow 0$
as $c_i' \rightarrow 0$. 

Let $E$ be the subalgebra of $D \otimes \R$ that centralizes
$x_{\infty} a x_{\infty}^{-1}$.  The assertion that
$\gamma \in  x_{\infty} B(C,\varepsilon) B(C, \varepsilon)^{-1} x_{\infty}^{-1}$
shows that $\alpha$ is ``close'' to $E$; in fact, it is clear that
$$\inf_{e \in E} \| \alpha - e\| \ll \varepsilon$$ 
and moreover $\|\alpha\| \ll  1$ (because $\alpha$ lies in a fixed compact subset
of $(D \otimes \R)^{\times}$). 

By Lemma \ref{diophantine2} we see that, if $c_1'', c_2''$ are sufficiently
small -- this occurs, in particular, if $c_1, c_2$ are sufficiently small --
then all such $\alpha$ necessarily belong to a proper $\Q$-subalgebra of $D$;
because $D$ has prime degree, this must be a field $F$.
Analyzing this reasoning shows that $c_2$ may be taken to depend only on
the isomorphism class of $\G$, whereas $c_1 = O(1)$.
This proves \eqref{intersect2}. 

Also, there exists $K \in \mathbb{N}$ so that
$D_{\Z}. D_{\Z} \subset K^{-1} D_{\Z}$.
Then $\alpha. D_{\Z} \subset K^{-1} . \denom(\alpha)^{-1} . D_{\Z}$
and similarly for $D_{\Z}. \alpha$. 

Replacing $\alpha$ with $\alpha' = K. \denom(\alpha). \alpha$, we see that
$\alpha'D_{\Z}+D_{\Z}\alpha' \subset D_{\Z}$
and
$\|\alpha'\| \leq c_3 \varepsilon^{-c_4}$, where
$c_3 = O_{C}(1)$ and $c_4$ depends only on the isomorphism class of $\G$. 
\end{proof}

We also need to know that there are only a few bad primes.
\begin{lem} \label{badprimes}
Let $F \subset D$ be a subfield, and let $\T_F \subset \G$ be the torus
defined by $F$, \ie the centralizer of $F$ in $\G$.  Suppose that $F$
is generated, over $\Q$, by an element $\alpha$ satisfying
$\alpha D_{\Z} + D_{\Z} \alpha \subset D_{\Z}$. Then 
the number of primes which fail to be $\T_F$-good is at most
$O(1+ \log \|\alpha\|)$. 
\end{lem}

\begin{proof} 
Let $\Kf'$ be the stabilizer\footnote{ Note that $\GAf$
  acts naturally on lattices inside $D$, in a fashion derived from
  the conjugation action of $\G$ on $D$. Indeed, if $V$ is a $\Q$-vector
  space, the group $\GL(V \otimes_{\Q} \Af)$ acts naturally on lattices
  inside $V$.}
of $D_{\Z}$ inside $\GAf$.

Then $\mathbf{T}_F(\Q_p) \cap \Kf'$ is maximal compact inside
$\mathbf{T}_F(\Q_p)$ as long as the maximal compact subring of 
$(F \otimes \Q_p)$ preserves $D_{\Z} \otimes \Z_p$ under both
left and right multiplication.   This will be so, in particular,
at any prime where the maximal compact subring of $(F \otimes \Q_p)$
equals $\Z_p[\alpha]$. This will always be the case if $p$ does not
divide the discriminant of the ring $\Z[\alpha]$.
 
From this, one deduces that there at most $O(1+\log\|\alpha\|)$
primes for which $\mathbf{T}_F(\Q_p) \cap \Kf'$ fails to be maximal
compact in $\mathbf{T}_F(\Q_p)$.  But, for all but $O(1)$ primes,
the intersection $\G(\Q_p) \cap \Kf'$ coincides with $\G(\Q_p) \cap \Kf$.
So there are at most $O(1+\log\|\alpha\|)$ primes for which
$\mathbf{T}_F(\Q_p) \cap \Kf$ fails to be maximal compact in
$\mathbf{T}_F(\Q_p)$. 
\end{proof}  
 
\section{Bounds on the mass of tubes, II}  \label{sec:tubemass2}



We define ``tubes'' $B_0 := B(C,\varepsilon)$ as in Section
\ref{sec:dioph}.  When $\G$ is merely assumed $\R$-split we take
$C \subset MA$.  In the special case of division algebras of prime
degree $C$ may be taken to lie in any Levi subgroup of $G$.

Set $B = B_0 B_0^{-1} \supset B(C,\varepsilon)$.
There exists a compact subset $C' \subset Z_G(a)$ and a constant $M=O(1)$
such that $B$, $B_2$ $(=B_1.B_1)$ and $B_3$ are subsets of
$B(C',M \varepsilon)$. Here notations are as \eqref{balli}. 
Also, $\frac{\vol B_3}{\vol B_0} \ll 1$.\footnote{ In this and in the
  statement $M=O(1)$, the implicit constant certainly depends on $C$;
  however, $C$ was assumed to belong to $\Omega_{\infty}$, and we have
  permitted implicit constants to depend on $\Omega_{\infty}$ without
  explicit mention.}

\subsection{Sets $S$ for which intersections are controlled by tori}\label{entrop}
Let $Q$ be so that $Q/2$ is larger than any bad prime for $\G$,
and let $\ell$ be fixed. (In practice, $Q\to\infty$ 
as $B$ becomes small, whereas $\ell$ is fixed depending only on $\G$). 

Set $S_p = \{g_p \in \G(\Q_p)/K_p: \denom_p(g_p) \leq p^{\ell}\}$.
Initially we shall consider the set of translates given by
$\bigcup_{p \in [Q/2,Q]} S_p$, where we identify $\G(\Q_p)/K_p$ 
with a subset of $\GAf/\Kf$ in the natural way.

Part (1) of the conlusions of Propositions \ref{propertystar1} and
\ref{propertystar2} can now be rephrased\footnote{
  Note that, in what follows, we are applying the Proposition
  with $B(C,\varepsilon)$ replaced by the larger set
  $B_2 \subset B(C', M. \varepsilon)$. It is easy to see that the
  stated result holds even though this larger set may not be contained
  in $\Ominf$.}
as establishing (for $Q^l \ll \varepsilon^{-c_2}$) the following condition.
In words, it states that {\em intersections between Hecke translates of $B_2$ 
by $\cup S_p$ all arise from a $\Q$-torus:}

\framebox{$\star = \star_{B, Q, \ell}$}:
For any $x \in \Omega_{\infty}$ there is a
$\Q$-torus $\T \subset \G$ so that 
$$ \overline{xB_2 s} \cap \overline{xB_2 s'} \neq \emptyset  \ \ \mbox{ in $X$}$$
with $s, s' \in \bigcup_{p \in [Q/2,Q]} S_p$
only if there is $\gamma \in \T(\Q)$ so that for these $s,s'$,
$$\gamma xB_2 s \cap xB_2 s' \neq \emptyset \mbox{ in $\GA/\Kf$}.$$

\begin{lem} \label{multbound}
Suppose that condition $\star_{B, Q, \ell}$ is satisfied.
Take $x \in \Omega_{\infty}$ and let $\T$ be the torus specified
by $\star_{B,Q,\ell}$.
Assume that $S \subset \bigcup_{p} S_p \subset \G(\adele_f)/\Kf$,
with the union taken over $p \in [Q/2,Q]$ which are $\T$-good.   Then:

$$\#\{s, s' \in S: \overline{xB_2  s} \cap \overline{ xB_2 s'} \neq \emptyset \} 
 \ll_{\ell} Q^2 +|S| $$
\end{lem} 

\begin{proof} 
Consider any intersection $\overline{x B_2s} \cap \overline{ xB_2 s' }\neq \emptyset$ in $X$ when $s, s' \in S$. 
This means that there is $\gamma \in \T(\Q)$ so that
$\gamma xB_2 s  \cap x B_2 s'  \neq \emptyset$ in $\GA/\Kf$. 
Then $s \in S_p, s' \in S_q$ when $p,q$ are $\T$-good primes in the range $[Q/2,Q]$; we distinguish two cases according to whether
$p=q$ or not. 

\begin{enumerate}
\item 
$q=p$.  In this case,
$$\gamma \in (x B_2 B_2^{-1} x^{-1}). \T(\Q_p) . \Kf.$$

For any fixed $p$,  the number of $\Kf$-cosets contained in $\T(\Q_p) . \Kf$  
satisfying $\denom_p \leq \ell$ is $O_\ell(1)$: since $p$ is $\T$-good, the quotient $\T(\Q_p) / \T(\Q_p) \cap \Kf$
is a free abelian group of rank $\leq \dim(\T)$. Pick generators $t_1, \dots, t_r$
for this quotient; they generate a discrete subgroup. We need to show that the number of $(e_1, \dots, e_r) \in \Z^r$
so that $\denom_p(t_1^{e_1} \dots t_r^{e_r}) \leq p^{\ell}$ is $O_{\ell}(1)$. To see this,
pass to an extension of $\Q_p$ where $\rho(t_i)$ become diagonalizable.

Therefore, $\gamma$ is an element of $\G(\Q)$ so that:
\begin{enumerate}
\item considered as an element of $\G(\R)$, $\gamma$
belongs to $$x B_2^{-1} B_2 x^{-1},$$
which in turn is contained in a compact set depending only on $\Omega_{\infty}$; 
\item considered as an element of $\GAf$, $\gamma$ belongs to
$O_{\ell}(1)$ right $\Kf$-cosets.
\end{enumerate}

The number of possibilities for $\gamma$ is therefore $O_{\ell}(1)$.  Since $(s K_f,s' K_f)$ is determined by $(s K_f,\gamma)$, it follows the number of possibilities for $(s,s')$ in
the case ``$p=q$'' is at most $O_{\ell}(|S|)$.

\item $p \neq q$. 

In this case, $s \in \mathbf{T}(\Q_p). K_p$
and $s' \in \mathbf{T}(\Q_q). K_q$. By an argument already given, the number of $K_p$-cosets contained in $\T(\Q_p). K_p$ and
satisfying $\denom_p \leq \ell$ is $O_{\ell}(1)$, and similarly with $q$ replacing $p$.
It follows that the number of possibilities for $(s,s')$ is $O_{\ell}(1)$ for {\em given} $p,q$.

The total number of possibilities for $(s,s')$ in the case ``$p \neq q$'' is therefore $O_{\ell}(Q^2)$. 
\end{enumerate} 
\end{proof}

\subsection{Conclusion} 
We shall apply Lemma \ref{Plancherel} to our setting. Take $\ell$ as in that Lemma.

\begin{lem} \label{conclusion}

Suppose that condition $\star_{B, Q, \ell}$ is satisfied.
Take $x \in \Omega_{\infty}$ and let $\T$ be the torus specified
by $\star_{B,Q,\ell}$, $\mathcal{P}$ the set of $\T$-good primes in $[Q/2,Q]$.

For any  Hecke eigenfunction $\psi$ on $X$. 
$$\mu_{\psi}(\barxB) \ll \frac{ Q^{1.01}}{|\mathcal{P}|^{2}}.$$
\end{lem} 

\begin{proof}
For $p \in \mathcal{P}$ let $h_p$ be the $K_p$-bi-invariant function
on $G_p$ furnished in Lemma \ref{Plancherel}. 
Now, if we consider $h_p$ as a function on $G_p/K_p$, 
we have upper and lower bounds $p \ll \# \supp(h_p) \ll p^{\ell'}$,
where $c$ depends only on $\G, \rho$. 

Therefore, by a dyadic decomposition argument, there exists
a subset $\mathcal{P}_1 \subset \mathcal{P}$ and $A \gg Q$
so that $\# \supp(h_p) \in [A,2A]  \ \ (p \in \mathcal{P}_1)$. 
Set $h = \sum_{p\in\mathcal{P}_1} h_p$. 
Then, by what we have proven in Lemma \ref{cover} and Lemma \ref{multbound}, 
$$\mu_{\psi}(\barxB) \ll
 \left( \frac{\vol B_3}{\vol B_0} \right)^2
 \frac{ Q^2 + \sum_{p} \# \supp(h_p) }{ \left( \sum_p \# \supp(h_p)^{1/2} \right)^2 }
\ll \left( \frac{\vol B_3}{\vol B_0} \right)^2 Q^{0.001}( \frac{Q}{|\mathcal{P}|^2} +g^{-1}).$$

Finally, as observed at the start of this section, $\frac{\vol B_3}{\vol B_0} \ll 1$. 
\end{proof} 

We now combine this Lemma with the results of Section \ref{sec:dioph},
which give conditions under which $\star_{B,Q,\ell}$ is true. 

\begin{thm} \label{thm:tubemass}
Let $\G$ be a semisimple group defined over $\Q$ which splits over $\R$.
Let $\psi$ be a Hecke eigenfunction on $X=\GQ\bs\GA/\Kf$.

Let $\Omega_{\infty} \subset G=NAK$ be compact.  Further, let
$B(C,\varepsilon) \subset \Omega_{\infty}$ be a tube
as in Section \ref{sec:dioph} such that either

\begin{enumerate}
\item $C \subset MA$; or,

\item $\G$ is the projectivized group of units of a division algebra $D$
of prime degree over $\Q$.
\end{enumerate}

Then there is $c > 0$,
depending only on the isomorphism class of $\G$ so that, uniformly over
$x \in \Omega_{\infty}$,
$$\mu_{\psi}(x B(C,\varepsilon))  \ll \varepsilon^{c}.$$
\end{thm}

\begin{proof} 
Let $\ell$ be as in Lemma \ref{Plancherel}.

Recall, as remarked at the start of the present section, that
$B_3 \subset B(C', M. \varepsilon)$, for a suitable compact set
$C'$ and a suitable constant $M$.

Proposition \ref{propertystar1} (for case (1)) or \ref{propertystar2}
(for case (2)), applied to $B(C', M. \varepsilon)$\footnote{
instead of $B(C, \epsilon)$; it is easy to see the proof works verbatim
even though $B(C', M. \varepsilon)$ need not be contained in $\Ominf$},
shows that property $\star_{B, Q, \ell}$ holds so long as
\begin{equation} \label{qconstraint}
Q^{\ell} \leq  a \epsilon^{-b},
\end{equation}
where $a,b$ depend only the isomorphism class of $\G$. 

If $Q$ satisfies this constraint, the previously quoted 
Propositions, in combination with Lemma \ref{badprimes}, 
show that the number of primes that are not $\T$-good is $O(\log \varepsilon)$. 
(Here $\T$ is the torus occurring in the definition of $\star_{B, Q, \ell}$.)

This shows, notation as in Lemma \ref{conclusion}, that $|\mathcal{P}| \gg \frac{Q}{\log Q}$. 
Thus $\mu_{\psi}(x B(C,\varepsilon)) \ll Q^{-0.98}$.
Choosing $Q$ as large as allowable under \eqref{qconstraint}
yields the desired result. 
\end{proof}

\section{The AQUE problem and the application of the entropy bound.}
\label{sec:results}
We now return to the AQUE problem discussed in the Introduction,
recall our previous work on this problem, and explain how our main
theorem concerning AQUE is deduced. 
\subsection{Quantum unique ergodicity on locally symmetric spaces}

\begin{problem}
\label{pro: SQUE-Y}(QUE on locally symmetric spaces; Sarnak) Let
$G$ be a connected semi-simple Lie group with finite center. Let
$K$ be a maximal compact subgroup of $G$, $\Gamma<G$ a lattice,
$X=\Gamma\bs G$, $Y=\Gamma\bs G/K$.
Let $\left\{ \psi_{n}\right\} _{n=1}^{\infty}\subset L^{2}(Y)$
be a sequence of normalized eigenfunctions of the ring of $G$-invariant
differential operators on $G/K$, with the eigenvalues \wrt the Casimir
operator tending to $\infty$ in absolute value.  Is it true that
$\bar{\mu}_{\psi_{n}} := |\psi_n|^2 d\vol$
converge weak-{*} to the normalized projection of the Haar measure to $Y$? 
\end{problem}

In the paper \cite{SilbermanVenkatesh:SQUE_Lift} we have obtained
Theorem \ref{mainthmA}, recalled below, constructing the microlocal
lift in this setting.  We needed to impose a non-degeneracy condition
on the sequence of eigenfunctions
(the assumption essentially amounts to asking that all eigenvalues
tend to infinity, at the same rate for operators of the same order.)
For the precise definition of {\em non-degenerate}, we refer to 
\cite[Section 3.3]{SilbermanVenkatesh:SQUE_Lift}. 

With $K$ and $G$ as in Problem \ref{pro: SQUE-Y}, let $A$ be as
in the Iwasawa decomposition $G=NAK$, \ie $A=\mathrm{exp}(\mathfrak{a})$
where $\mathfrak{a}$ is a maximal abelian subspace of $\mathfrak{p}$. For
$G=\PGL_d(\R)$ and $K=\PO_d(\R)$, one may take $A$ to be the
subgroup of diagonal matrices with positive entries.
Let $\pi\colon X\rightarrow Y$ be the projection.
We denote by $dx$ the $G$-invariant probability
measures on $X$, and by $dy$ the projection of this measure to $Y$.

\begin{thm}
\label{mainthmA} Let $\psi_{n}\subset L^{2}(Y)$ be a
non-degenerate sequence of normalized eigenfunctions, whose eigenvalues
approach $\infty$. Then, after replacing $\psi_{n}$ by an appropriate
subsequence, there exist functions $\tilde{\psi}_{n}\in L^{2}(X)$
and distributions $\mu_{n}$ on $X$ such that: 
\begin{enumerate}
\item \label{thmA:claim1} (Lift) The projection of $\mu_{n}$ to $Y$ coincides
with $\bar{\mu}_{n}$, \ie $\pi_{*}\mu_{n}=\bar{\mu}_{n}$.
\item \label{thmA:claim2} Let $\sigma_{n}$ be the measure
$|\tilde{\psi}_{n}(x)|^{2}dx$ on $X$. Then, for every $g\in\D$,
we have $\lim_{n\rightarrow\infty}(\sigma_{n}(g)-\mu_{n}(g))=0$.
\item \label{thmA:claim3} (Invariance) Every weak-{*} limit
$\sigma_{\infty}$ of the
measures $\sigma_{n}$ (necessarily a positive measure of mass $\leq1$)
is $A$-invariant. 
\item \label{prop:eq} (Equivariance). Let $E\subset\mathrm{End}_{G}(\CX)$
be a $\C$-subalgebra of bounded endomorphisms of $\CX$, commuting
with the $G$-action. Noting that each $e\in E$ induces an endomorphism
of $C^{\infty}(Y)$, suppose that $\psi_{n}$ is an eigenfunction
for $E$ (\ie 
$E\psi_{n}\subset\C\psi_{n}$). Then we may choose $\tilde{\psi}_{n}$
so that $\tilde{\psi}_{n}$ is an eigenfunction for $E$ with the
same eigenvalues as $\psi_{n}$, \ie for all $e\in E$ there exists
$\lambda_{e}\in\C$ such that $e\psi_{n}=\lambda_{e}\psi_{n},e\tilde{\psi}_{n}=\lambda_{e}\tilde{\psi}_{n}$. 
\end{enumerate}
\end{thm}
We first remark that the distributions $\mu_{n}$ (resp. the measures
$\sigma_{n}$) generalize the constructions of Zelditch (resp. Wolpert).
Although, in view of (\ref{thmA:claim2}), they carry roughly equivalent
information, it is convenient to work with both simultaneously: the
distributions $\mu_{n}$ are canonically defined and easier to manipulate
algebraically, whereas the measures $\sigma_{n}$ are patently positive
and are central to the arguments of the present paper.

The existence of the microlocal lift already places a restriction
on the possible weak-{*} limits of the measures $\left\{ \bar{\mu}_{n}\right\} $
on $Y$. For example, the $A$-invariance of $\mu_{\infty}$ shows
that the support of any weak-{*} limit measure $\bar{\mu}_{\infty}$
must be a union of maximal flats.  Following Lindenstrauss, we term
the weak-{*} limits $\sigma_\infty$ of the lifts $\sigma_n$
\emph{quantum limits}.

More importantly, Theorem \ref{mainthmA} allows us to pose a new
version of the problem:

\begin{problem}
\label{pro: SQUE-X}(QUE on homogeneous spaces) In the setting of
Problem \ref{pro: SQUE-Y}, is the $G$-invariant measure on $X$
the unique non-degenerate quantum limit?
\end{problem}

The main result of this paper is the resolution of Problem \ref{pro: SQUE-X}
for certain higher rank symmetric spaces, in the context of \emph{arithmetic}
quantum limits. We refer to \cite[Section 1.4]{SilbermanVenkatesh:SQUE_Lift} for a further discussion of the significance of these spaces and how the introduction of arithmetic helps to eliminates degeneracy. 

\subsection{Results: Arithmetic QUE for division algebra quotients} \label{finalresults}
For brevity, we state the result in the language of automorphic forms;
in particular, $\mathbb{A}$ is the ring of ad{\`e}les of $\Q$. 

Let $\G$ be a semisimple group over $\Q$, and let $G = \G(\R)$.
Let $\Kf$ be an open compact subgroup of $\GAf$ such that
$X=\GQ\backslash\GA/\Kf$ consists of a single $G$-orbit
(this condition is mainly cosmetic: see Remark \ref{Eichler}
in \S \ref{Remarksection}).  Then there exists a discrete subgroup
$\Gamma<G$ such that $X=\Gamma\backslash G$. 
Let $\mathcal{H}$ be the Hecke algebra, as defined in
Section \ref{sec:notation}. It acts on $L^2(X)$. 
Set $Y = \Gamma \backslash G/K$ the associated locally symmetric space,
where $K$ is the a maximal compact subgroup inside $G$. 
$A$ will denote a maximal $\R$-split torus of in $G$ compatible with $K$.

In the special case, let $D/\Q$ be a division algebra of prime degree $d$,
and let $\G$ be the associated projective general linear group,
i.e. the quotient of the group of units in $D$ by its center.
Assume that $\G$ is $\R$-split, \ie $G=\G(\R)\isom\PGL_d(\R)$ (when $d\geq 3$
this is always the case).  For this group let $K$ be the
standard maximal compact subgroup, $A$ the group of diagonal matrices
with positive entries (up to scaling).

Theorem \ref{thm:tubemass} implies:

\begin{thm}
\label{mainthmB} Let $\tilde{\psi}_{n}\in L^{2}(X)$ be
a sequence of $\mathcal{H}$-eigenfunctions on $X$ such that the
associated probability measures $\sigma_{n}:= |\tilde{\psi}_n(x)|^2 dx$
on $X$ converge weak-{*} to an $A$-invariant probability measure
$\sigma_{\infty}$.
Then every \emph{regular} $a\in A$ acts on every
$A$-ergodic component of $\sigma_{\infty}$ with positive entropy.
When $\G$ is associated to a division algebra, the same holds for any
$a\in A\setminus\{1\}$.
\end{thm}
\begin{proof}
This is essentially a rephrasing of Theorem \ref{thm:tubemass}, where
the uniformity of the estimate means it carries over to weak-{*} limits.

For a proof that the bound on measures of tubular neighbourhood of
$Z_G(a)$ implies that $a\in A$ acts with positive entropy
see \cite[Sec.\ 8]{Lindenstrauss:AdelicDyn_AQUE}. While written for
the case of quaternion algebras ($d=2$), that discussion readily
generalizes to our situation by modifying its ``Step 2'' to account
for the action of $a$ on the Lie algebra.
\end{proof}

Using results on measure-rigidity due to Einsiedler and Katok,
this has the following implication for the QUE problem: 

\begin{thm}
\label{thm:que} Let $\G$ be the projectivized unit group of a division
algebra of prime degree, and maintiain the other notations as above.
Let $\left\{ \psi_{n}\right\} _{n=1}^{\infty}\subset L^{2}(Y)$
be a non-degenerate sequence of eigenfunctions for the ring of
$G$-invariant differential operators on $G/K$
(\cf \cite[Sec.\ 3.3]{SilbermanVenkatesh:SQUE_Lift})
which are also eigenforms of the Hecke algebra $\mathcal{H}$
(\cf Section \ref{sec:notation}).
Such $\psi_n$ are also called Hecke-Maass forms.

Then the associated probability measures $\bar{\mu}_{n}$ converge
weak-{*} to the normalized Haar measure on $Y$, and their lifts
$\mu_{n}$ (see Theorem \ref{mainthmA}) converge weak-{*} to the
normalized Haar measure $dx$ on $X=\Gamma\backslash\PGL_{d}(\R)$.
\end{thm}
\begin{proof}
The case $d=2$ is Lindenstrauss's theorem, and we will thus assume $d\geq3$. 

Passing to a subsequence, let $\psi_{n}\in L^{2}(Y)$
be a non-degenerate sequence of Hecke-Maass forms on $Y$ such that
$\bar{\mu}_{n}\to\bar{\mu}_{\infty}$ weakly.  Passing to a subsequence,
let $\tilde{\psi}_{n}$ and $\sigma_{n}$ be as in Theorem \ref{mainthmA}
such that $\sigma_{n}\to\sigma_{\infty}$ weakly and $\sigma_{\infty}$
lifts $\bar{\mu}_{\infty}$. Then $\sigma_{\infty}$ is a non-degenerate
arithmetic quantum limit on $X$.  By Theorem \ref{mainthmB},
$\sigma_{\infty}$ is an $A$-invariant probability measure on $X$
such that every $a\in A\setminus\{1\}$ acts on almost every $A$-ergodic
component of $\sigma_{\infty}$ with positive entropy.
Remark also that the measure $\sigma_{\infty}$ is invariant under the
finite group $Z_K(A)$, the centralizer of $A$ in $K$, by construction
(see \cite[Remark 1.7, (5)]{SilbermanVenkatesh:SQUE_Lift}). 

Then \cite[Thm.\ 4.1(iv)]{EinsiedlerKatok:SLn_Rigid}
shows that $\sigma_{\infty}$ has a unique ergodic component,
$\mu_{\textrm{Haar}}$.
\end{proof}

Our methods also apply to the case of a split central simple $\Q$-algebra,
that is when $\G(\Q) = \PGL_d(\Q)$.  The result is somewhat weaker, however:

\begin{thm}
\label{thm:que2} Let $G = \PGL_d(\R)$ ($d$ prime), and let $\Gamma < G$
be a lattice of the form \footnote{
  Again, the result is best expressed for an ``adelic'' quotient}
$\PGL_d(\Z) \cap \gamma \PGL_d(\Z)\gamma^{-1}$ for some
$\gamma \in \PGL_d(\Q)$. Let $X=\Gamma \backslash G$, $Y=X/K$
Let $\left\{ \psi_{n}\right\} _{n=1}^{\infty}\subset L^{2}(Y)$
be a non-degenerate sequence of Hecke-Maass forms.

Then the associated probability measures $\bar{\mu}_{n}$ converge
weak-{*} to the a Haar measure on $Y$, and their lifts
$\mu_{n}$ (see Theorem \ref{mainthmA}) converge weak-{*} to
a Haar measure $c dx$ on $X=\Gamma\backslash\PGL_{d}(\R)$,
where $c\in [0,1]$.
\end{thm}

\begin{proof}
Let $\G = \PGL_d/\Q$.
For any prime $p$ let $\mathcal{O}_p$ be the (``Eichler'') order
$M_n(\Z_p) \cap \gamma M_n(\Z_p)\gamma^{-1}$ of $M_n(\Q_p)$.
Let $K_p = \mathcal{O}_p^\times$, $\Kf = \prod_p K_p$.  Then
$\Gamma = \PGL_d(\Q)\cap \Kf$, so that $X = \GQ \backslash \GA / \Kf$.

Passing to a subsequence, let $\mu$ be a weak-* limit of a sequence of lifts.
Theorem \ref{mainthmB} shows that there exist $a\in A$ which act with
positive entropy on almost every $A$-ergodic component of $\mu$.
The measure rigidity results of \cite{EKL:SLn_Rigid} together
with the orbit classification results of 
\cite{LindenstraussWeiss:SLn_Algebraic} (we use here the fact that $d$ is
prime) show that $\mu$ is a Haar measure on $X$.  Since $X$ is not compact,
this method does not control the total mass $c$ of $\mu$.
\end{proof}

\subsection{Remarks on generalizations} \label{Remarksection}

\subsubsection{Class number one} \label{Eichler}
The assumption imposed that $G$ act with a single orbit in
$\G(\Q) \backslash \GA /\Kf$ is, as we remarked, cosmetic.
In general, if we remove this assumption, 
one would still know -- making analogous definitions -- that quantum limits
remain $G$-invariant.  However, this would not quite be a complete answer
since the space of $G$-invariant measures on $\GQ\bs\GA/\Kf$ is now finite
dimensional, and we would not know the relative measures of the different
components. 

\subsubsection{Nondegeneracy}  The second author has obtained a version
of Theorem \ref{mainthmA} without the non-degeneracy assumption, see
\cite{Silberman:SQUE_DegenLift_preprint}.  In that case the lifts are
asymptotically invariant under (non-trivial) subgroups of $A$.
The bounds on the mass of tubes obtained in this paper are, at their
foundation, purely statements about Hecke eigenfunctions, and thus carry
over to degenerate limits.  However, the interpretation of these bounds
as lower bounds of the \emph{entropy} only applies to tubes associated
to elements $a\in A$ by which the measure is invariant.

In the degenerate case, a-priori one only has invariance by \emph{singular}
elements $a\in A$.  Thus our methods only show that those elements act
with positive entropy in the case of division algebras.

An analogue of Theorem \ref{thm:que} would accordingly follow from
a result classifying measures on $X$ which are only invariant by a
(potentially one-dimensional) subgroup of the Cartan subgroup, assuming
a Hecke recurrence condition a-la \cite{Lindenstrauss:SL2_QUE}.
An advance in this direction is necessary in order to show that
every sequence of Hecke-Maass forms is equidistributed; in the rest of the
remarks we only consider the case of non-degenerate limits.

It is worth noting that the mass of tubes correpsonding to the flow of a
regular element would also be small, but since the measure is not invariant
by the regular element it is not clear how to incorporate this information
into the measure-classification result.

\subsubsection{Escape of mass.}
When the quotient $X$ is not compact (for example, in the split case of
Theorem \ref{thm:que2}), there is an
additional potential obstruction to equidistribution: weak-* limits are
not necessarily probability measures -- they may even be the zero measure.
This possibility is known as ``escape-of-mass''.

In the case of congruence lattices in $\SL_2(\Z)$, escape-of-mass was
ruled out by Soundararajan \cite{Sound:EscapeOfMass}.  This was generalized
to the congruence lattices in $\SL_2(\mathcal{O})$, $\mathcal{O}$ the ring of
integers of a number field, in the M.Sc.\ Thesis \cite{Zaman:EscapeThesis}.
The extension to higher-rank groups is the subject of current research.

In the particular case of congruence lattices in
$\GL_d(\Z)$ the normalization of the measure is already controlled by
the degenerate Eisenstein series.  Hence a sub-convexity result for
the Rankin--Selberg $L$-function would control the escape, just as in the
better-known case of $\GL_{2}$ (though, notably, Soudararajan's argument
does not rely on an L-function bound). In fact, such a sub-convexity result
would also prove that the limits have \emph{positive entropy} (this
is demonstrated in \cite{ELMV3}).

In the rest of our remarks we ignore the issue as well;
the reader may assume the group $\G$ to be anisotropic.

\subsubsection{The case when $\G$ is not associated to a division algebra.}
We expect the techniques developed for the proof of Theorem \ref{thm:que}
will generalize at least to some other locally symmetric spaces, the
case of division algebras of prime degree being the simplest; but there are
considerable obstacles to obtaining a theorem for \emph{any} arithmetic locally
symmetric space at present.  A brief discussion of some of these difficulties
follows. 

First, the intersection patters of Hecke translates will be controlled
by subgroups more complicated than tori.  Except for the simplest case,
in general the best one can hope for is that intersections
be controlled by Levi subgroups defined over $\Q$.  Lemma \ref{diophantine2}
already establishes this for unit groups of semisimple $\Q$-algebras.
Such subgroups will have \emph{exponential} volume growth
(say in terms of their orbit on the building of $\G(\Q_p)$), compared with
the polynomial behaviour of tori.  Even the purely local question of
whether an eigenfunction on a building can concentreate appreciably
on the orbit of such a large subgroup is difficult.
To see where such an issue can arise note that when $G$ is not $\R$-split,
even the centralizer of the maximal $\R$-split torus is not a torus.

Dealing with intersections created by larger $\Q$-subgroups
is essential for a more fundamental reason.
The best possible outcome of the type of measure classification results one
would use here (for state of the art see
\cite{EinsiedlerLindenstrauss:GeneralMaxSplit_preprint})
is that the measure is, in some sense, \emph{algebraic}: it is a linear
combination of measures supported on orbits of subgroups.
From this point of view, to show that the limit measure is the $G$-invariant
measure should at least require showing that the mass of orbits of these
subgroups is zero a-la Rudnick--Sarnak.  In our terms, this means
showing that the mass concentrated in an an $\epsilon$-neighbourhood of the
orbit goes to zero with $\epsilon$ even if it is not necessary
to achieve power-law decay of the mass (\ie positive entropy).
In fact, for this reason it is hard to imagine an application of the current
techniques that would rule out these intermediate measures without also
estalishing that all elements of $A$ act with positive entropy.

\appendix
\section{Proof of Lemma \ref{Plancherel}: how to construct a higher
rank amplifier} 

To readers familiar with the usage of ``amplification'' in analytic number
theory (as represented, for instance, in the work of Duke, Friedlander and
Iwaniec): the Lemma \ref{Plancherel} in effect represents a way to construct
an amplifier in higher rank.   

Let $\G$ be a semisimple algebraic group over $\Q$, $\rho\colon\G\to\SL_N$
an embedding.  For each prime $p$ let $G_p := \G(\Q_p)$.  There is
$p_0$ such that for $p>p_0$, $K_p = \rho^{-1}(\SL_N(\Z_p))$ is a special
maximal compact subgroup of the unramified group $G_p$.

We will allow the implicit constant in the symbols $\gg,\ll$ to depend
on $N$ (and hence, also on dimension of $\G$), but not on anything else. 
\begin{lem} \label{Plancherel}
Possibly inreasing $p_0$ there exist integers $\ell,\ell'$ depending only
on $N$ such that for any $p>p_0$ and for any character $\Lambda$ of the Hecke
algebra of $G_p$ with respect to $K_p$, there is a $K_p$-bi-invariant function
$h_p$ with the following properties:
\begin{enumerate}
\item Its support satisfies
$ p^{\ell'} \gg  \# \supp(h_p) \gg  p$, where we think of $h_p$ as a function on $G_p/K_p$; 
\item $|h_p| \in \{0,1\}$; 
\item $\Lambda(h_p)$ is positive and $\Lambda(h_p) \gg  \# \supp(h_p)^{1/2}$
\item For any $g \in \supp(h_p)$, the denominator of $\rho(g)$
is $\ll_N p^{\ell}$.
\end{enumerate}
\end{lem} 

\subsection{Notation on $p$-adic groups.} \label{notationonpadic}
 
We shall use certain standard properties of semisimple algbraic groups over
$p$-adic fields.  Standard references are \cite{Tits:BuildSurv}
and \cite{Cartier:padicGpSurv}.

Increasing $p_0$ we may suppose that $G_p$ quasi-split and unramified,
and that $K_p$ hyperspecial.  Let $\mathbf{A}$ be a maximal $\Q_p$-split
torus in $G_p$ so that the corresponding apartment in the building of
$G_p$ contains the point fixed by $K_p$, and let $A_p = \mathbf{A}_p(\Q_p)$
be the corresponding subgroup of $G_p$.

Let $X_* = \mathrm{Hom}(\mathbb{G}_m, \mathbf{A})$ and
$X^* = \mathrm{Hom}(\mathbf{A}, \G_m)$.
Let $\Phi \subset X^*$ be the set of roots for the action of $\mathbf{A}$
on the Lie algebra of $\G$.  We note that, as $p$ varies, there will be at
most finitely many distinct root systems.  In particular, our bounds
may depend on $\Phi$.

Fix a positive system of roots for $A_p$, let $N_p$ be the subgroup
corresponding to all the positive roots. We have Iwasawa decomposition
$G_p = N_p . A_p . K_p$.  Let $\delta: A_p \rightarrow \R^{\times}$
be the character corresponding to the half-sum of positive roots,
composed with $\|\cdot\|_p$ on $\Q_p$. 

Let $\mathfrak{a} := A_p/ (A_p \cap K_p)$, a free abelian group of rank
equal to the rank of $\G(\Q_p)$.  Then $\mathfrak{a}$ is identified with
$X_*$: for, given $a \in \mathfrak{a}$, there exists a unique homomorphism
$\theta: \mathbb{G}_m \rightarrow \mathbf{A}_p$ so that
$\theta(p)$ and $a$ lie in the same $A_p \cap K_p$ coset.   

Next, let $V = X_*\otimes_\Z \R$, $V^* = X^* \otimes_\Z \R$,
and $V^*_\C = V^*\otimes_\R \C$.  Then to any $\nu \in V^*_\C$, and any
$\theta \in X_*$ with $a = \theta(p)$ we set
$a^{\nu} = p^{\left<\theta,\nu\right>}$ with the obvious pairing.  In
particular this gives an identification between the unramified characters of
$A_p$ and the torus $\mathfrak{a}_\temp^* = i V^* / (2\pi i \log p X^*)$.

Let $W$ be the Weyl group of $A_p$.  It acts on $\mathfrak{a}$. 
Moreover, we fix a $W$-invariant inner product on $\mathfrak{a} \otimes \R$
in the following way: the elements of $\Phi$, considered as belonging
to $\Hom(\mathfrak{a},\Z)$, define elements of a root system.  We require
the longest root for each simple factor of $\G$ to have length $1$.
This uniquely normalizes a $W$-invariant inner product on the dual to
$\mathfrak{a} \otimes \R$, so also on $\mathfrak{a} \otimes \R$. 

This normalization has the following property:  if $\alpha \in \Phi$
is any root and $a \in \mathfrak{a}$, then ``$|\alpha(a)|$'' -- we implicitly
identify $a$ with an element of $X_*$ -- is bounded above by $p^{\|a\|}$. 

Finally, let $M_p = Z_{K_p}(A_p)$ so that $N_p A_p M_p$ is a Borel subgroup
of $G_p$.

\subsection{Plancherel formula}

To any character $\nu: \mathfrak{a} \rightarrow \C^{\times}$ (parametrized
by an element of $V^* / (2\pi i \log p X^*)$), we associate
the spherical representation\footnote{Recall a spherical (irreducible) representation of $G_p$ is one that possesses a one dimensional space
of $K_p$-invariants}
$\pi(\nu)$ of $G_p$ obtained by extending $\nu \delta$ to
$N_p A_p M_p$ trivially on $N_p M_p$, inducing to $G_p$, and taking the
unique spherical subquotient. 

For any $K_p$-bi-invariant function $k$ on $G_p$, let $\hat{k}(\nu)$
be the scalar by which $k$ acts on the spherical vector in $\pi(\nu)$.
There is a unique $K_p$-bi-invariant function $\Xi_{\nu}$ on $G_p$
(the spherical function with parameter $\nu$) so that:
\begin{equation} \label{inverse}
\hat{k}(\nu) = \int_{g \in G_p} k(g) \Xi_{\nu}(g) dg, \ \ \
 k(g) = \int_{\nu\in\mathfrak{a}_{\temp}^*} \hat{k}(\nu) \Xi_{\nu}(g) d\mu(\nu).
\end{equation} 
where the first integral is taken \wrt the Haar measure that assigns
mass $1$ to $K_p$, and the second integral is taken \wrt the Plancherel
measure $\mu$ on $\mathfrak{a}_{\temp}^*$.
In our normalization, $\mu$ is a probability measure. 

The map $k \mapsto \hat{k}(\nu)$ is an isomorphism between the space of
compactly supported, $K_p$-bi-invariant functions on $\G(\Q_p)$, and the
space of $W$-invariant ``trigonometric polynomials'' on $\mathfrak{a}$;
here ``trigonometric polynomial'' means ``finite linear combination of
characters.'' Also $k \mapsto \hat{k}$ is an isometry:
\begin{equation} \label{l2norm}
\int_{\G(\Q_p)} |k(g)|^2 dg = \int_{\nu \in \mathfrak{a}_{\temp}^*} 
                                 |\hat{k}(\nu)|^2 d\mu(\nu)
\end{equation}

The explicit form of the Plancherel measure is known
\cite{Macdonald:SphericalFunctionsBook}.  From it we extract the following
fact: ``$\mu_p = \mu_{\infty} +O(p^{-1/2})$.''
More precisely, there exists a measure $\mu_{\infty}$ on 
$V^* / 2\pi X^*$ (whose Fourier transform is supported in
$\{X \in \mathfrak{a}: \|X\| \leq 3\}$) such that the difference
$\mu_p - \mu_{\infty}$ is a signed measure represented by a
function of supremum norm $O(p^{-1/2})$.  Here to consider $\mu_\infty$
as a measure on $\mathfrak{a}_\temp^*$ we identify with the our reference
torus $V^* / 2\pi X^*$ via rescaling by $\log p$.

\subsection{The Paley-Wiener theorem}

Recall our fixed $W$-invariant inner product on $\mathfrak{a}$.
The Paley-Wiener theorem asserts that under the transform $k \mapsto \hat{k}$, 
the preimage of characters of $\mathfrak{a}_{\temp}^*$ supported in
$\{X \in \mathfrak{a}: \|X\| \leq R\}$
is contained in functions supported in
$K_p . \{X \in \mathfrak{a}: \|X\| \leq R\}  . K_p$

Let us briefly sketch the proof.  Let $\mathfrak{a}^+$ be the
closed positive Weyl chamber within $\mathfrak{a}$. 
For $\alpha, \beta \in \mathfrak{a}^+$, so that $\nu \mapsto \nu(\beta)$
occurs with a nonzero coefficient in $\widehat{K_p \alpha K_p}$,
then, necessarily, $\alpha - \beta$ belongs to the dual cone to
$\mathfrak{a}^+$.

Now, choose $\alpha \in \mathfrak{a}^+$ so that $k(K_p \alpha K_p) \neq 0$
and $\|\alpha\|$ is maximal subject to that restriction.
We claim that $\nu \mapsto \nu(\alpha)$ necessarily occurs in $\hat{k}$.
For, in view of the remarks above, if this were not the case there must
exist $K_p \beta K_p$ in the support of $k$,
with $\beta \in \mathfrak{a}^+$ and $\|\alpha\| < \|\beta\|$. This is
a contradiction.

\subsection{Proof of the amplification lemma}
Let $\nu_0 \in \mathfrak{a}^*$ be the parameter of the character
$\Lambda$ of the Hecke algebra specified in the Lemma.  We do not, of course,
assume that $\nu_0 \in \mathfrak{a}_{\temp}^*$). 

Take any $a \in \mathfrak{a}$ with $a$ not in the support of the Fourier
transform of $\mu_{\infty}$, but $\|a\|$ reasonably small. For example,
we could take $a$ to be twice the coroot associated to any root of maximal
length; then $\|a\| = 4$.  Take $R = |W| \| a\|$. 
Construct the function $k_1$ with spherical transform
\begin{equation} \label{elf}
\hat{k}_1 (\nu) = \sum_{|j| \leq |W|} 
        \overline{ \sum_{w \in W} w \nu_0(j a)}  \sum_{w \in W}   w\nu( ja )
\end{equation} 

Note that if $\alpha_1, \dots, \alpha_m$ are any
nonzero complex numbers, then by a simple compactness argument.
\begin{equation} \label{compactness}
\max_{j \neq 0, |j| \leq m} \left| \alpha_1^j + \dots + \alpha_m^j \right|
             \geq c(m) > 0
\end{equation} 

So $M := \sum_{|j|\leq|W|} \left| \sum_{w \in W} w\nu_0(j a) \right|^2 \gg 1$.
We have $\sup_{\nu \in \mathfrak{a}_{\temp}^*} |\hat{k}_{1}(\nu)| \ll M^{1/2}$. 
Thus, by \eqref{l2norm},
$\|k_{1}\|_{L^2}^2 \ll  M$; by definition, $\hat{k}_{1} (\nu_0) = M$;
and from the explicit form of Plancherel measure and \eqref{inverse},
 $|k_{1}(1)| = O(M^{1/2} p^{-1/2})$. 

Put $k = k_{1} - k_{1}(1) 1_{K_p}$. Then $k(1) =0$ and  --
if we suppose that the residue field size, $p$, is sufficiently large,
as we may --  $|\hat{k}(\nu_0)| \gg  \|k\|_{L^2}$.

By the Paley-Wiener theorem, $k$ is supported in
$\bigcup_{|a| \leq R} K_p a K_p$.  The number of $K_p$-double cosets in this
set is equal to the number of $a \in \mathfrak{a}$ with $|a| \leq R$,
which is, in turn $O(1)$. ($a$ is completely determined
by the value of all the roots on it; but the number of roots is $O(1)$
and each value is an integer $\leq R$).  We conclude that there is
$|a| \leq R$ so that 
\begin{equation} \label{one}
|\widehat{1_{K_p a K_p}} (\nu_0) |^2 \gg \int_{K_p a K_p} dg.
\end{equation} 
On the other hand, it is known (\cite[Section 3.5]{Cartier:padicGpSurv}) that
\begin{equation} \label{two}
 \int_{K_p a K_p} dg  \asymp \delta(a)^2
\end{equation}
 the notation $\asymp$ meaning that the quotient is bounded above and below,
at least for $p \geq p_0(\dim \G)$.  Since $\delta$ is the half-sum of
positive roots, and each root has length $\leq 1$, we see that
$p \leq \delta(a)^2 \leq p^{\dim(\G) R}.$
(\cf the end of \S \ref{notationonpadic}.)
 
We take $h_p$ to be the multiple $1_{K_p a_i K_p}$ by a suitable complex
number of absolute value $1$, so that $\hat{h}_p(\nu_0)$ is positive.
The first three assertions of the lemma follow from remarks already made. 

We now turn to establishing the necessary bounds on the denominator of
$\rho(g)$.  Let $\Phi_{\rho} \subset X^*$ be the weights of the representation
of $\rho$ with reference to $\mathbf{A}$; for each $\alpha \in \Phi_{\rho}$,
let $V(\alpha)$ be the weight space. We require two preliminary observations:

Firstly, $\sup_{\alpha \in \Phi_{\rho}} |\alpha| \leq N$.  To verify
this we may suppose that the representation $\rho$ is irreducible.
Because $\G$ is semisimple, we may choose $w \in W$ so that
$\|w \alpha - \alpha\| \geq \|\alpha\|$.  Call two elements of $X^*$
\emph{connected} if they differ by an element of $\Phi$. 
There is a sequence $\alpha = \alpha_0, \alpha_1, \dots, \alpha_{r} =w \alpha$
where $\alpha_i, \alpha_{i+1}$ are connected and
$\alpha_i \in \Phi_{\rho}$ for all $i$. 
So $\dim(V) \geq r +1 \geq \|w \alpha -\alpha\| \geq \|\alpha\|$. 

Secondly, $\Z_p^N = \oplus_{\alpha}  \Z_p^N \cap V(\alpha)$ so long as 
the restriction map from characters of
$X^*$ to $\Hom(A_p \cap K_p, \Z_p^{\times})$ is injective.
For, this being the case, the spaces $\Z_p^N \cap V(\alpha)$ are
characterized as the ``eigenspaces'' of a prime-to-$p$ finite group
acting on $\Z_p^N$.  This is so, in particular, if $p \geq p_0(N)$. 

Combining these remarks, we see at once that the denominator of $\rho(a)$
is $\leq p^{N R}$ when $a \in A_p$ projects to the ball of radius $\leq R$
in $\mathfrak{a}$.  The desired bound on denominators follow. \qed

\bibliographystyle{amsplain}
\bibliography{aut_forms,ergodic_theory,lie_gps,que}

\end{document}